\documentclass[10pt,a4paper]{amsart}

\usepackage{amsfonts}
\usepackage{amsthm}
\usepackage{amssymb}
\usepackage{latexsym}
\usepackage{amsmath}
\usepackage{setspace}
\usepackage{amscd}
\usepackage{verbatim}
\usepackage[pdftex]{graphicx}
\usepackage{graphics}
\usepackage[matrix,arrow]{xy}
\usepackage[active]{srcltx}
\usepackage{mathrsfs}

\usepackage
[hypertexnames=false,backref=page,pdftex,pdfpagemode=UseNone,breaklinks=true,extension=pdf,colorlinks=true,linkcolor=blue,
urlcolor=blue]{hyperref}

\title[Generic properties of $3$-dimensional Reeb flows]{Generic properties of $3$-dimensional Reeb flows: Birkhoff sections and entropy}

\author[V. Colin]{Vincent Colin}
\address{V. Colin, Nantes Universit\'e, CNRS, Laboratoire de Math\'ematiques Jean Leray, LMJL, F-44000 Nantes, France}
\email{vincent.colin@univ-nantes.fr}
\urladdr{https://www.math.sciences.univ-nantes.fr/~vcolin/}

\author[P. Dehornoy]{Pierre Dehornoy}
\address{P. Dehornoy, Aix Marseille Universit\'e, CNRS, I2M, Marseille, France}
\email{pierre.dehornoy@univ-amu.fr}
\urladdr{http://www-fourier.ujf-grenoble.fr/~dehornop/}

\author[U. Hryniewicz]{Umberto Hryniewicz}
\address{U. Hryniewicz, Chair of Geometry and Analysis, RWTH Aachen,
Jakobstrasse 2,
52064 Aachen, Germany}
\email{hryniewicz@mathga.rwth-aachen.de}
\urladdr{https://www.mathga.rwth-aachen.de/~hryniewicz/home/}

\author[A. Rechtman]{Ana Rechtman}
\address{A. Rechtman, Institut Fourier,
Universit\'e Grenoble Alpes,
100 rue des mathématiques 38610 Gières, France}
\email{ana.rechtman@univ-grenoble-alpes.fr}
\urladdr{https://www-fourier.ujf-grenoble.fr/~rechtmaa/}

\date{\today}

\newcommand{\C}{\mathbb{C}}
\newcommand{\R}{\mathbb{R}}
\newcommand{\Z}{\mathbb{Z}}
\newcommand{\N}{\mathbb{N}}
\newcommand{\D}{\mathbb{D}}
\newcommand{\Q}{\mathbb{Q}}
\newcommand{\supp}{{\rm supp}}

\newcommand{\s}{\mathscr{S}}
\newcommand{\Cc}{\mathscr{C}}
\renewcommand{\P}{\mathscr{P}}

\newcommand{\pr}{{\rm pr}}

\newcommand{\F}{\mathcal{F}}
\newcommand{\Leb}{{\rm Leb}}
\newcommand{\E}{\mathcal{E}}
\renewcommand{\span}{{\rm span}}
\newcommand{\vol}{{\rm vol}}
\newcommand{\lk}{{\rm Lk}}

\theoremstyle{plain}
\newtheorem{theorem}{\sc Theorem}[section]

\newtheorem{proposition}[theorem]{\sc Proposition}
\newtheorem{lemma}[theorem]{\sc Lemma}
\newtheorem{corollary}[theorem]{\sc Corollary}

\theoremstyle{definition}
\newtheorem{definition}[theorem]{\sc Definition}

\theoremstyle{remark}
\newtheorem{remark}[theorem]{\sc Remark}

\begin{document}

\maketitle

\begin{abstract}
In this paper we use broken book decompositions to study Reeb flows on closed $3$-manifolds.
We show that if the Liouville measure of a nondegenerate contact form can be approximated by periodic orbits, then there is a Birkhoff section for the associated Reeb flow.
In view of Irie's equidistribution theorem, this is shown to imply that the set of contact forms whose Reeb flows have a Birkhoff section contains an open and dense set in the $C^\infty$-topology.
We also show that the set of contact forms whose Reeb flows have positive topological entropy is open and dense in the $C^\infty$-topology.
\end{abstract}

\tableofcontents

\section{Introduction and results}

The notion of a Birkhoff section  (see Definition~\ref{defn_BS_GSS})  is a classical tool to study flows on $3$-manifolds,  going back to Poincar\'e. 
When a flow admits such a section, its dynamics can be reduced to the dynamics of the first-return map on the section -- a much simpler data.
Birkhoff sections are particularly convenient for studying topological and dynamical properties of the flow.
In contact topology, work of Giroux~\cite{giroux} implies that every contact structure on a closed $3$-manifold admits a supporting open book decomposition. 
In particular,  every contact structure is defined by a contact form whose Reeb vector field  has  a global surface of section, i.e. an embedded Birkhoff section. 
However  for a given contact form --~hence a given Reeb vector field~-- deciding if Birkhoff sections exist is difficult.

Not every flow on a closed $3$-manifold admits Birkhoff sections. 
For example a flow which is not a suspension and has no periodic orbits does not have one. 
Notice however that such a flow cannot be Reeb in view of the result of Taubes~\cite{taubes} asserting that every Reeb vector field on a closed 3-manifold has a periodic orbit.
On the other hand, many Reeb flows do admit Birkhoff sections.
A large class of examples is provided by a striking result due to Hofer, Wysocki and Zehnder~\cite{convex} which contains, as a special case, the fact that Hamiltonian flows on $3$-dimensional strictly convex compact energy levels possess   disk-like global surfaces of section.   
As a consequence, Hofer, Wysocki and Zehnder proved that on such an energy level there are two or infinitely many periodic orbits.
Another class consists of geodesic flows on closed, orientable and connected Riemannian $2$-manifolds such that the curvature has a definite sign. 
When the curvature is everywhere negative, this is implied by a result of Fried~\cite{friedanosov} asserting that transitive Anosov flows have Birkhoff sections.
When the curvature is everywhere positive this is covered by a classical result of Birkhoff~\cite{birkhoff}.

Consider a nonsingular vector field $X$ on a compact $3$-manifold $M$  tangent to~$\partial M$. 
If $\gamma:\R/T\Z \to M$ is a periodic orbit of period $T>0$, parametrized by the flow, then  $\frac{\gamma_*\Leb}{T}$ is an invariant Borel probability measure,  where $\Leb$ denotes the Lebesgue measure on $\R/T\Z$.
This measure is independent of the choice of period.
The set of invariant Borel probability measures is a convex subset of the topological dual of the space of continuous real-valued functions on~$M$ equipped with the supremum norm.
A Borel probability measure is said to be approximated by periodic orbits if it is the weak* limit of a sequence of finite convex combinations of measures induced by periodic orbits.

When $M$ is oriented, and $X$ is the Reeb vector field of a positive contact form~$\lambda$ with helicity $\vol(\lambda) = \int_M \lambda \wedge d\lambda>0$, then we say that the Liouville measure can be approximated by periodic orbits if $(\lambda \wedge d\lambda)/\vol(\lambda)$ can be approximated by periodic orbits as above.
A periodic orbit is nondegenerate if its transverse linearized Poincar\'e map has no root of unity as an eigenvalue.
The contact form is nondegenerate if every periodic Reeb orbit is nondegenerate.
Birkhoff sections and $\partial$-strong Birkhoff sections are defined in Section~\ref{sec_BBDs}; $\partial$-strong Birkhoff sections are, in particular, also Birkhoff sections.
Our main result reads as follows.

\begin{theorem}
\label{thm:BirkhoffSection}
If the Liouville measure of a nondegenerate contact form on a closed $3$-manifold can be approximated by periodic Reeb orbits, then the Reeb flow admits a $\partial$-strong Birkhoff section.
\end{theorem}

A result due to Irie~\cite{irie} asserts that the set of nondegenerate contact forms on a closed $3$-manifold whose Liouville measures can be approximated by periodic orbits is residual, in particular dense, with respect to the $C^\infty$-topology.
Together with Proposition~\ref{prop: open} we obtain the following statement.

\begin{corollary}
\label{cor_generic_existence}
On any closed $3$-manifold, the set of contact forms such that the Reeb flow admits a $\partial$-strong Birkhoff section spanned by nondegenerate periodic orbits is open and dense in the $C^\infty$-topology.
\end{corollary}


The main key ingredient of the proof of Theorem~\ref{thm:BirkhoffSection} is the existence result found in~\cite{CDR} for \textit{broken book decompositions} carrying the Reeb vector field of a nondegenerate contact form on any closed $3$-manifold.
These are special kinds of foliations  on  the complement of a finite set of periodic orbits, the so-called binding orbits, and  whose  leaves are transverse to the Reeb vector field. 
One of the main applications of broken book decompositions obtained in~\cite{CDR} is the statement that every nondegenerate Reeb flow on a closed $3$-manifold has either two or infinitely many periodic orbits, generalizing a previous result by Cristofaro-Gardiner, Hutchings and Pomerleano~\cite{CGHP} for cases where the first Chern class of the contact structure is torsion.
All these results rely on the machinery of Embedded Contact Homology as defined by Hutchings~\cite{hutchings}.

\begin{remark}
A result of Sigmund~\cite{sigmund} implies that the Liouville measure of an Anosov Reeb flow can be approximated by periodic orbits. 
This fact and Theorem~\ref{thm:BirkhoffSection} together prove that Anosov Reeb flows have Birkhoff sections. 
This is a special case of Fried's results from~\cite{friedanosov}.
\end{remark}

While writing this article, we learned that similar results were independently obtained by Contreras and Mazzucchelli~\cite{CM}.

\begin{theorem}[Contreras and Mazzucchelli~\cite{CM}]
\label{thm_CM}
The Reeb flow of a nondegenerate contact form on a closed $3$-manifold admits a Birkhoff section provided that the stable and unstable manifolds of all hyperbolic periodic Reeb orbits intersect transversely.
\end{theorem}

A contact form whose Reeb flow satisfies the assumptions of the above result will be called here strongly nondegenerate,  and were said to satisfy the Kupka-Smale condition in~\cite{CM}. 
Their argument takes the path suggested in \cite{CDR}, see Remark~\ref{remark: CM}, which is different from the one followed in this article.
Both ways still start from a common crucial input: the existence result for broken book decompositions from~\cite{CDR}.
 In Appendix \ref{section: C} we  sketch  an argument using the strategy elaborated in \cite{CDR} and prior to the present work as well as to \cite{CM}, but which leads to a generic existence result only valid in $C^1$-topology, see Theorem \ref{thm: section}. It relies on the machinery developed by Arnaud, Bonatti and Crovisier \cite{ABC,BC} and involves the proof of the {\it Lift Axiom} (Lemma~\ref{lemma: lift Axiom}) for Reeb flows.
In turn, the Lift Axiom and the derived connecting lemma (Theorem~\ref{thm: connecting}) also imply Theorem \ref{thm: transitive}: the set of transitive Reeb vector fields is $C^1$-generic, in dimension $2n+1$ with $n\geq 1$.

Broken book decompositions generalize the finite-energy foliations obtained for nondegenerate Reeb flows on the tight 3-sphere by Hofer, Wysocki and Zehnder in their pioneering work~\cite{fols}. The analysis from~\cite{fols} led to major breakthroughs in the study of Reeb dynamics, as well as in the development of pseudo-holomorphic curve theory in symplectic cobordisms.
In particular, Hofer, Wysocki and Zehnder were able to prove that a strongly nondegenerate Reeb flow on the tight $3$-sphere has two or infinitely many periodic orbits, leading them to first conjecture that the $2/\infty$ dichotomy must hold for all Reeb flows on the tight $3$-sphere.

The literature is full of results with sufficient  existence  conditions for Birkhoff sections, going back a century.
Let us comment on the problem of finding global sections spanned by \textit{given} periodic orbits, with no genericity assumptions. 
The simplest example is Birkhoff's theorem~\cite{birkhoff}.
Closed geodesics on oriented surfaces determine immersed annuli on the unit tangent bundle, made of unit vectors along the closed geodesic pointing ``to one of the sides''.
This is called a Birkhoff annulus.
It is proved in~\cite{birkhoff} that Birkhoff annuli over embedded closed geodesics on positively curved spheres are global surfaces of section for the geodesic flow.
This result is interesting because it gives strong dynamical conclusions out of geometric assumptions that can be checked in concrete examples.
The analogue of Birkhoff's result in the context of convex energy levels is found in~\cite{H}, where the periodic orbits that span disk-like global surfaces of section are characterized. 
The convexity assumption was first dropped in~\cite{HS}, eventually leading to the generalization of Birkhoff's result to Reeb flows via pseudo-holomorphic curves in~\cite{HSW}, where it is shown that certain linking assumptions on invariant measures, well-known to be sufficient for global sections of general flows (\cite[Theorem~1.3]{Hry_SFS}), can sometimes be reduced to linking assumptions on periodic orbits.

It is also interesting to try to measure the minimal complexity that a Birkhoff section of a Reeb flow can have. 
For instance, Fried showed that geodesic flows on hyperbolic surface always admit torus-like Birkhoff sections~\cite{friedanosov} and this was recently extended to hyperbolic 2-dimensional orbifolds~\cite{D, DS}. 
On the other hand, van Koert constructed Reeb flows on $S^3$ without disk-like global surfaces of section~\cite{vK}. 

\medskip

Our second result concerns the topological entropy of Reeb flows in dimension three.
It is already known that a nondegenerate Reeb vector field with zero topological entropy admits a Birkhoff section~\cite{CDR}. 
In the absence of elliptic periodic orbits the techniques  of Koropecki, Le Calvez and Nassiri~\cite{KLN}, and of Le Calvez and Sambarino~\cite{LS}, that we adapt to diffeomorphisms of surfaces with boundary, imply that the topological entropy of the flow is positive. When the flow has an elliptic periodic orbit, we can perturb the flow in a neighborhood of this orbit to obtain the following statement.

\begin{theorem}\label{thm: entropy}
Let $M$ be a closed $3$-manifold and $\xi \subset TM$ be a co-orientable contact structure. The set of contact forms on $M$ defining $\xi$ such that the Reeb flow has positive topological entropy is open and dense in the set of all contact forms on $M$ defining $\xi$, with respect to the $C^\infty$-topology.
\end{theorem}

\medskip

This result can be seen as a generalization of the fact that the geodesic flow on a closed surface has positive topological entropy $C^\infty$-generically on the metric. For surfaces of negative Euler characteristic the geodesic flow always has positive topological entropy regardless of the Riemannian metric, as follows from the results by Dinaburg \cite[Section 4]{Dinaburg}. In the case of the torus, when the geodesic flow is nondegenerate (or equivalently, the Riemannian metric is {\it bumpy}) there is a collection of length minimizing non-intersecting closed geodesics and by Theorem XVI in \cite{Hedlund}, there is a heteroclinic cycle between them. If the intersections along this cycle are not transverse, the arguments of Donnay \cite{Donnay} allow to locally perturb the metric to make them transverse, thus obtaining a geodesic flow with positive topological entropy. The difficult case of the sphere is treated by Knieper and Weiss \cite{KW}. Our proof follows essentially the same path as in \cite{KW}:
 they used the pseudo-holomorphic curves of Hofer, Wysocki and Zehnder~\cite{fols} to produce Birkhoff sections, dealing with the extra constraint of perturbing the metric instead of the flow among its Reeb neighbors.

There are many results available in the literature providing conditions that force a Reeb flow to have positive topological entropy.
In~\cite{CDR} it was proved that if the $3$-manifold is not graphed --~for example if it is hyperbolic~-- then every nondegenerate Reeb vector field has positive topological entropy. 
Symplectic topological methods for studying positive topological entropy of Hamiltonian systems have now been vastly used, going back to Polterovich~\cite{polterovich}, Frauenfelder and Schlenk~\cite{FS} and others.
In~\cite{ACH} a condition on the fractional Dehn twist coefficients of a supporting open book decomposition is given to guarantee that every Reeb vector field, possibly degenerate, for the given supported contact structure has positive topological entropy.
Alves and Pirnapasov~\cite{AP} proved that any contact $3$-manifold admits transverse links that force positive topological entropy when realized as periodic Reeb orbits.
Alves and Meiwes obtain in~\cite{AM} contact structures in higher dimensional spheres such that the Reeb flows of all defining contact forms have positive topological entropy.

\medskip

We end this introduction with a rough outline of the proof of Theorem~\ref{thm:BirkhoffSection}.

Given a flow in a closed manifold~$M$ and given a class $y$ in $H^1(M)$, Schwartzman~\cite[Section 7]{schwartzman} gave a necessary and sufficient condition for the flow to admit a Birkhoff section with empty boundary dual to~$y$: for every invariant Borel probability measure~$\mu$ the intersection of $\mu$ with $y$, see Section~\ref{ssec_invt_measures_int_numbers} for a definition, must be positive. 
This criterion can be adapted to the existence of a Birkhoff section with prescribed boundary, via Schwartzman-Fried-Sullivan theory~\cite{fried,schwartzman,sullivan}. 
It was already known to specialists for some time but we provide an explicit statement as Theorem~\ref{thm_existence_from_SFS}: given a link $L$ made of periodic orbits and a class $y$ in $H^1(M\setminus L)$, there exists a $\partial$-strong Birkhoff section if every invariant measure supported in~$M\setminus L$ intersects $y$ positively and on every component of~$L$ the flow winds positively with respect to~$y$.
A refinement for global surfaces of section can be found in~\cite{Hry_SFS}.

It is shown in~\cite{CDR} that when a contact form is nondegenerate its Reeb vector field is supported by a broken book decomposition $(K,\F)$ in a special manner. 
Here $K$ is a link formed by periodic orbits, called binding orbits, and $\F$ is a special smooth foliation of $M\setminus K$, see Definition~\ref{defn_BB} for a precise description. 
The binding splits into special sublinks $K = K_r \cup K_b$, the components $K_b$ are the so-called broken binding orbits.
It follows from definitions that when $K_b = \emptyset$, the broken book is a rational open book decomposition whose pages are Birkhoff sections for the Reeb flow.
Hence we proceed assuming that $K_b \neq \emptyset$.
 
In this case the foliation~$\F$ contains a finite set of special leaves, called rigid leaves, see Section~\ref{ssec_rigid_leaves}, which almost meet the conditions of Theorem~\ref{thm_existence_from_SFS}: their union~$S_{R_\F}$ intersects all orbits of the flow. 
However it is not true that the flow winds positively with respect to~$S_{R_\F}$, see Figure~\ref{figure: brokencomponent1} right. 
In order to enforce this property, one would like to find a surface transverse to the flow, bounded by periodic orbits in $M\setminus K_b$, and that intersects all components of $K_b$  positively. 
Then one could ``add it'' to the surface $S_{R_\F}$, thus obtaining the desired Birkhoff section, using a process that can be traced back to Fried's work~\cite{friedanosov}, see also~\cite{CDR}.

We do not find directly such a transverse surface. 
However we remark that it is enough to find a surface bounded by periodic orbits and which intersects the broken binding orbits positively --- the point being that this surface need not be positively transverse to the flow in its interior as is required in Fried's process.
Indeed, denoting by $S_b$ this surface, we remark that the ``sum'' $nS_{R_\F}+S_b$ intersects positively all invariant measures when $n$ is large enough, and that the flow winds positively with respect to it along the broken binding orbits.

At the technical level, it is easier to find the Poincar\'e dual to this surface $S_b$ : 
we will find an additional link of periodic orbits $K' \subset M \setminus K_b$ and a cohomology class $y' \in H^1(M\setminus K';\R)$ which essentially plays the role of the Poincar\'e dual of~$S_b$.
Once this task is done, Theorem~\ref{thm:BirkhoffSection} is a direct consequence of Theorem~\ref{thm_existence_Birkhoff_sections}, stated in Section~\ref{sec_BBDs} and proved in Section~\ref{sec_getting_BSs}.
It uses the fine structural dynamical information revealed by the broken book decomposition in order to check the assumptions of Theorem~\ref{thm_existence_from_SFS}.

We are then left with the task of finding $K'$ and $y'$, which is the content of Proposition~\ref{prop_gen_1}.
For this we need to establish a small modification of the Action-Linking Lemma from~\cite{BSHS}.
This lemma relates the intersection number between a given surface and the Liouville measure (a topological data) to the contact area of the surface.
Hence, if the Liouville measure can be approximated by periodic orbits then the contact area is re-obtained as the limit of intersection numbers of weighted links of periodic orbits with the surface.
The link $K'$ is found among those links.

In the special case where $M$ is a homology $3$-sphere, we can make the argument explicit and short: denoting by $I_k$ a sequence of weighted links that approximate the Liouville measure, and by $h_1, \dots, h_n$ the broken binding orbits, the Action-Linking Lemma implies that the linking $\lk(h_i, \lambda\wedge d\lambda)$ equals the action $T(h_i)$, which is positive. 
Since $\lk(h_i, \lambda\wedge d\lambda)$ is the limit of the sequence $\lk(h_i, I_k)$ when $k\to\infty$, we obtain that for $k$ large enough, all linking numbers $\lk(h_i, I_k)$, $i=1, \dots, n$, are positive. 
This implies that~$I_k$ links positively with all broken binding orbits. 
Therefore we can take $K'=I_k$ and $y'=\lk(\cdot,I_k)$.

This argument is special to homology $3$-spheres, in general it needs to be replaced by a refined one found in Section~\ref{sec_general_scheme}.
Once these tools are in place, Theorem~\ref{thm:BirkhoffSection} has the following simple proof.

\begin{proof}
[Proof of Theorem~\ref{thm:BirkhoffSection}]
By~\cite[Theorem~1.1]{CDR} we know that the Reeb vector field of a nondegenerate contact form is strongly carried by a broken book decomposition, see Definition~\ref{defn_BB} and Remark~\ref{rmk_strongly_carried}. An application of Proposition~\ref{prop_gen_1} with $L=K_b$ tells us that the hypotheses of Theorem~\ref{thm_existence_Birkhoff_sections} are satisfied.
\end{proof}

\begin{remark}
\label{remark: CM}
Contreras and Mazzucchelli studied in~\cite{CM} the closure of the invariant manifolds of the broken binding orbits, and found many homoclinic connections. 
Once enough such homoclinic connections are found,   the strategy described in \cite{CDR} and also implemented in the proof of Theorem~\ref{thm: section} applies. It is possible to show that there are additional periodic orbits spanned by transverse surfaces that intersect the broken binding orbits.
In~\cite[Theorem 4.13]{CDR} it is then explained that these additional surfaces can be ``added'' to form a new broken book decomposition carried by the contact form, but with strictly less broken binding orbits.
After a finite number of steps, one gets a rational open book decomposition whose pages are Birkhoff sections.
A key contribution from~\cite{CM}, which seems surprising and of independent interest, is that under the strong nondegeneracy assumption the closure of stable and unstable manifolds of the broken binding orbits coincide.
Using this fact, the desired additional transverse surfaces can be constructed and the strategy described above can be applied.
At this point of the argument, these additional surfaces provide precisely the data $K',y'$ and the reduction process can be replaced by our Theorem~\ref{thm_existence_Birkhoff_sections}.
\end{remark}

\subsection*{Acknowledgments}
Theorem~\ref{thm: section} and the use of connecting lemmas were suggested to us by Fran\c cois B\'eguin. 
The implementation of these in the proof of Theorem~\ref{thm: section} benefited from the clarifications of Sylvain Crovisier.
The proof of Theorem~\ref{thm: entropy} was suggested to us by Sobhan Seyfaddini, and some details were explained by Patrice Le Calvez.
We warmly thank them all for their time, patience and expertise.
Agustin Moreno and Abror Pirnapasov had also independently noticed that the finite-energy foliations from~\cite{fols}, combined with the results by Le Calvez and Sambarino, imply $C^\infty$-genericity of positive entropy for Reeb flows on the tight $3$-sphere.
We thank the Institut Henri Poincar\'e, where we had a working space during the trimester ``Symplectic topology, contact topology and interactions'' and the Oberwolfach Institute where part of these results were discussed.
We acknowledge partial support by the DFG SFB/TRR 191 ‘Symplectic Structures in Geometry, Algebra and Dynamics’, Projektnummer 281071066-TRR 191 and by the ANR grants Gromeov, IdEx UGA, Quantact, COSY and CoSyDy.

\section{Broken book decompositions}
\label{sec_BBDs}

\subsection{Rotation numbers along periodic orbits}
\label{ssec_rot_numbers}

Let $M$ be any orientable smooth $3$-manifold, and let $X$ be a smooth vector field on $M$.
Consider a periodic orbit $\gamma \subset M \setminus \partial M$ of $X$, with primitive period $T>0$, as a knot oriented by $X$. 
This orientation and the ambient orientation together co-orient~$\gamma$. 
Denote by $\phi^t$ the local flow near~$\gamma$.
Consider a small compact neighborhood~$N$ of $\gamma$, and an orientation preserving diffeomorphism
\begin{equation}
\label{tubular_map}
\Psi: N \to \R/T\Z \times \D
\end{equation}
satisfying $\Psi(\phi^t(p_0)) = (t,0)$ for some $p_0\in\gamma$.
Here the closed unit disk $\D\subset\C$ and the circle $\R/T\Z$ are oriented by the canonical orientations of $\C$ and $\R$ respectively, and $\R/T\Z \times \D$ gets the product orientation.
On $N\setminus\gamma$ we have coordinates
\begin{equation}
\label{polar_tubular_coordinates}
(t,r,\theta) \in \R/T\Z \times (0,1]\times\R/2\pi\Z, \qquad \Psi^{-1}(t,re^{i\theta}) \simeq (t,r,\theta)
\end{equation}
loosely referred to as tubular polar coordinates around $\gamma$. 
Consider the vector bundle $E_\gamma = TM|_{\gamma} / T\gamma \to \gamma$ oriented by the co-orientation of $\gamma$, and denote by $E_\gamma^*$ the complement of its zero section.
The total space of the circle bundle $E_\gamma^*/\R_+ \to \gamma$ is a torus, which can be equipped with global coordinates $(t,\theta) \in \R/T\Z \times \R/2\pi\Z$ induced by $\Psi$. 
The linearized flow $D\phi^t$ along $\gamma$ descends to a flow on $E_\gamma^*/\R_+$ represented as the flow of a vector field of the form
\begin{equation}
\label{form_of_ODE_linearized_flow}
\partial_t+b(t,\theta)\partial_\theta
\end{equation}
on this torus. 
The smooth function $b(t,\theta)$ is $(T\Z\times 2\pi\Z)$-periodic.

\begin{definition}
If $y \in H^1(N \setminus \gamma;\R)$ is cohomologous to $p \, dt + q \, d\theta$, with constants $p,q \in \R$, then define the rotation number of $\gamma$ relative to $y$ as
\begin{equation}
\label{def_formula_rot_number}
\rho^y(\gamma) = \frac{T}{2\pi} \left( p+q\lim_{t\to+\infty}\frac{\theta(t)}{t} \right)
\end{equation}
where $\theta:\R\to\R$ is any solution of $\dot\theta(t)=b(t,\theta(t))$.

\end{definition}

\begin{remark}
The rotation number $\rho^y(\gamma)$ does not depend on the choice of $\Psi$, or on the choice of the solution $\theta(t)$; see~\cite[Section~2]{Hry_SFS} for details.
\end{remark}

\begin{remark}
If $\gamma\subset M \setminus \partial M$ as above is a component of some link $L \subset M$ and $y \in H^1(M\setminus L;\R)$, then we may identify $y$ with the element of $H^1(N \setminus \gamma;\R)$ obtained by pulling $y$ back via the inclusion map $N \setminus \gamma \hookrightarrow M\setminus L$. We still write $\rho^y(\gamma)$ for the corresponding rotation number, with no fear of ambiguity.
\end{remark}

\subsection{Broken books and Birkhoff sections}

Consider a closed, connected, orientable and smooth $3$-manifold $M$ with a smooth nonsingular vector field $X$.
The flow of $X$ is denoted by $\phi^t$.

\begin{definition}
\label{defn_BS_GSS}
A section for the flow of $X$ is an immersion $\iota : S \to M$ defined on a compact surface $S$ such that:
\begin{itemize}
\item[(i)] If $\partial S \neq \emptyset$ then $\iota(\partial S)$ is a link consisting of periodic orbits of $X$. 
\item[(ii)] $\iota^{-1}(\iota(\partial S)) = \partial S$ and $\iota$ defines an embedding $S\setminus \partial S \hookrightarrow M\setminus \iota(\partial S)$ transverse to $X$.
\end{itemize}
A Birkhoff section, or rational global surface of section, for the flow of $X$ is a section such that:
\begin{itemize}
\item[(iii)] For every $p\in M$ there exist $t_-<0<t_+$ such that $\phi^{t_\pm}(p) \in \iota(S)$.
\end{itemize}
If $\iota$ is simultaneously an embedding and a Birkhoff section, then $\iota(S)$ is called a global surface of section for the flow of $X$.
\end{definition}

If $\gamma$ is a periodic orbit of $X$ then $E_\gamma = TM|_\gamma/T\gamma$ is a vector bundle over~$\gamma$. 
Hence $\mathbb{P}^+\gamma = (E_\gamma \setminus 0)/\R_+ \to \gamma$ is a circle bundle.
Note that the linearized flow $D\phi^t|_\gamma$ determines a smooth flow on $\mathbb{P}^+\gamma$.
We call this flow the linearized flow on $\mathbb{P}^+\gamma$, with no fear of ambiguity. 
If $\iota:S\to M$ is a section for the flow of $X$ and $c$ is a connected component of $\partial S$ such that $\iota(c) =\gamma$, then $\iota$ defines a smooth map $\nu^c_\iota:c \to \mathbb{P}^+\gamma$ as follows: choose any smooth vector field $n$ of $S$ along $c$ pointing outwards, so that $d\iota(n)$ is a map $c \to TM|_\gamma$, and define $\nu^c_\iota$ to be the map obtained by composing $d\iota(n)$ with the quotient map $TM|_\gamma \setminus T\gamma \to \mathbb{P}^+\gamma$.
The definition of $\nu^c_\iota$ does not depend on the choice of $n$.
The trace of $\iota$ along $c$ is defined as the image of the map $\nu^c_\iota$, in particular it is a subset of $\mathbb{P}^+\gamma$.

\begin{definition}
Let $\iota:S\to M$ be a section for the flow of $X$.
\begin{itemize}
\item We call $\iota$ a $\partial$-strong section if its trace along every connected component $c \subset \partial S$ is an embedding transverse to the linearized flow on $\mathbb{P}^+\gamma$; here $\gamma$ is the periodic orbit in $\iota(\partial S)$ that contains $\iota(c)$.
\item If $\iota$ is a Birkhoff section then we call $\iota$ a $\partial$-strong Birkhoff section if it is a $\partial$-strong section, and if for every connected component $c \subset \partial S$ its trace along $c$ defines a global section for the linearized flow on $\mathbb{P}^+\gamma$; here $\gamma$ is the periodic orbit that contains $\iota(c)$.
\end{itemize}
\end{definition}

\begin{remark}
When the $\partial$-strong Birkhoff section is a global surface of section, then we get a $\partial$-strong global surface of section as in~\cite[Definition~1.6]{FH}.
\end{remark}

\begin{definition}[\cite{CDR}]
\label{defn_BB}
Assume now that $M$ is oriented.
The vector field $X$ is said to be strongly carried by a broken book decomposition $(K,\F)$, where the binding $K \subset M$ is a link consisting of periodic orbits and $\F$ is a smooth foliation of $M \setminus K$, if the following conditions are satisfied.
\begin{itemize}
\item[(I)] The binding $K$, which is oriented by the flow, splits into two sublinks as $K = K_r\sqcup K_b$, where $K_r$ is called the radial part of the binding and $K_b$ the broken part of the binding. 
The leaves of $\F$ are transverse to and co-oriented by $X$, and oriented by the orientation of $M$ and this co-orientation. 
In particular, trajectories intersect the leaves positively.
\item[(II)] For every leaf $\ell$ of $\F$ there exists a compact connected oriented surface~$S$ with non-empty boundary, and a 
section $\iota : S \to M$ for the flow such that $\iota|_{S\setminus\partial S}$ defines an orientation preserving diffeomorphism $S \setminus \partial S \to \ell$, and $\iota|_{\partial S}$ defines a (not necessarily surjective) submersion $\partial S \to K$. 
Let $c$ be a connected component of $\partial S$ with $\iota(c) = \gamma \subset K$, in which case we say that $\gamma$ is in the boundary of $\ell$. 
Let $k \in \Z\setminus\{0\}$ be the degree of $\iota|_c:c\to \gamma$, and let $\ell^* \in H^1(M\setminus K;\R)$ be the class dual to $\ell$.
\begin{itemize}
\item[(a)] If $\gamma \in K_r$ then $\rho^{\ell^*}(\gamma)>0$.
\item[(b)] If $\gamma \in K_b$ then $|k| \in \{1,2\}$, $\rho^{\ell^*}(\gamma)=0$, $\gamma$ is hyperbolic and its transverse linearized Poincar\'e map has real eigenvalues $\alpha<\beta$. 
If $|k|=1$ then $\gamma$ is positive hyperbolic in the sense that $0<\alpha<1<\beta$. 
If $|k|=2$ then $\gamma$ is negative hyperbolic in the sense that $\alpha<-1<\beta<0$.
\end{itemize}
\item[(III)] If $\gamma \subset K_r$ then the intersection of the leaves of $\F$ with a small disk $D$ transverse to $\gamma$ defines a radial foliation of $D \setminus \gamma$ centered at $D \cap \gamma$. 
\item[(IV)] If $\gamma \subset K_b$ then the intersections of the leaves of $\F$ with a small disk $D$ transverse to $\gamma$ divide $D\setminus \gamma$ into eight  sectors centered at $D\cap \gamma$. 
Four of these sectors do not intersect $W^s(\gamma) \cup W^u(\gamma)$, are radially foliated and might have empty interior. 
These are intercalated by four open sectors containing $(W^s(\gamma) \cup W^u(\gamma)) \cap (D\setminus \gamma)$, which are foliated by hyperbola. 
\end{itemize}
\end{definition}

\begin{figure}[ht]
\centering
\includegraphics[width=.3\textwidth]{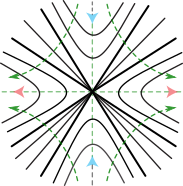}\quad
\includegraphics[width=.3\textwidth]{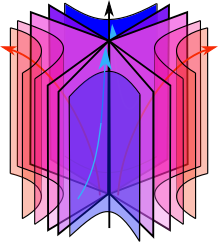}\qquad
\includegraphics[width=.3\textwidth]{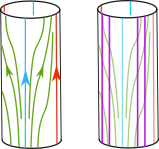}
\caption{A broken book decomposition strongly carrying a vector field $X$ in a neighbourhood of a broken binding component~$\gamma$. 
On the left, the transversal view:  the intersection of the broken book with a disc transverse to~$X$. 
The special leaves are bolded. 
Some orbits of~$X$ are represented with green dotted lines. 
In particular the four local stable/unstable manifolds lie in four different hyperbolic sectors.
In the center the 3d-view. 
On the right, the dynamics on the blown-up broken binding orbit $\mathbb{P}^+\gamma$ and the special leaves (purple):
they are transverse to the projectivized flow and intersect all its orbits, except those corresponding to the stable and unstable manifolds.
This manifests the fact that the rotation number of the flow along a broken binding component with respect to the union of the special leaves $\rho^{\ell^*}(\gamma)$ is zero.}
\label{figure: brokencomponent1}
\end{figure}

\begin{remark}
\label{rmk_strongly_carried}
In~\cite{CDR} it is defined that a contact form is carried by the broken book decomposition $(K,\F)$ if all the properties in Definition~\ref{defn_BB} hold for the Reeb vector field $X$, except for (IIa). 
Without further assumptions, one gets a non-strict inequality $\rho^{\ell^*}(\gamma) \geq 0$ for all $\gamma \subset K_r$ and for all leaves $\ell$. 
But if $\rho^{\ell^*}(\gamma) = 0$ for some $\gamma \subset K_r$ and some leaf $\ell$ that has $\gamma$ on the boundary, then some iterate of $\gamma$ is degenerate in the sense that $1$ is an eigenvalue of the corresponding transverse Poincar\'e map.
Hence, if $\lambda$ is a nondegenerate contact form carried by $(K,\F)$ as in~\cite{CDR}, then its Reeb vector field is strongly carried by $(K,\F)$ as in Definition~\ref{defn_BB}.
\end{remark}

\begin{remark}
The finite-energy foliations obtained by Hofer, Wysocki and Zehnder for a nondegenerate Reeb flow on the standard contact $3$-sphere are $\R$-invariant foliations of $\R \times S^3$ whose leaves are pseudo-holomorphic curves with finite Hofer energy, all of which have genus zero and precisely one positive puncture. 
They induce a broken book decomposition $(K,\F)$ as above with several special properties.
The $\R$-invariant cylinders over the orbits of $K$ are precisely the $\R$-invariant leaves of the finite-energy foliation, the other leaves project to the leaves of $\F$. 
Moreover, all orbits in $K$ have self-linking number equal to $-1$, all orbits in $K_b$ are positive hyperbolic, and each of the four radially foliated sectors in (IV) consists of a single ray. There are several other special additional properties.
The reader is referred to~\cite[section~1.4]{fols}.
\end{remark}

\begin{theorem}
\label{thm_existence_Birkhoff_sections}
Suppose that $X$ is strongly carried by a broken book decomposition $(K = K_b \sqcup K_r,\F)$, and that there exists a link $K' \subset M \setminus K_b$ consisting of periodic orbits and a class $y' \in H^1(M \setminus K';\R)$ satisfying $\left< y',\gamma \right> > 0$ for every $\gamma \subset K_b$. Then there is a $\partial$-strong Birkhoff section for the flow of $X$ with boundary in $K \cup K'$.
\end{theorem}

Let $\lambda$ be a positive contact form on $M$. Consider the space $C^0(M)$ of continuous real-valued functions on $M$. It becomes a Banach space with the supremum norm. 
As described in the introduction, we shall say that the Liouville measure  can be approximated by periodic orbits if there exists a sequence of weighted links of periodic Reeb orbits 
\begin{equation*}
(\{\gamma^n_j\},\{p^n_j\}) \qquad \qquad j=1,\dots,N(n)
\end{equation*}
where each $\gamma^n_j$ is a periodic Reeb orbit with primitive period $T(\gamma^n_j) > 0$ and the weights~$p^n_j \in (0,1]$ satisfy $\sum_j p^n_j = 1$, such that 
$$ 
\lim_{n\to\infty} \sum_j p^n_j \frac{(\gamma^n_j)_*\Leb}{T(\gamma^n_j)} = \frac{\lambda \wedge d\lambda}{{\rm vol}(\lambda)} \, . 
$$
Here the orbits are seen as maps $\gamma^n_j:\R/T(\gamma^n_j)\Z \to M$ parametrized by the flow, $\Leb$ denotes Lebesgue measure on $\R/T(\gamma^n_j)\Z$, ${\rm vol}(\lambda) = \int_M \lambda \wedge d\lambda$, and the limit is taken in the weak* topology of the topological dual $C^0(M)'$ of $C^0(M)$.

\begin{proposition}
\label{prop_gen_1}
Suppose that $X$ is the Reeb vector field of a positive contact form on $M$ such that $\lambda\wedge d\lambda$ can be approximated by periodic Reeb orbits.
Let $L \subset M$ be a link consisting of periodic Reeb orbits.
Then there exists a link $K' \subset M \setminus L$ and $y' \in H^1(M\setminus K';\R)$ such that $K'$ consists of periodic Reeb orbits, and $\left< y',\gamma \right> > 0$ for every $\gamma \subset L$.
\end{proposition}

\section{From linking with the broken binding to Birkhoff sections}
\label{sec_getting_BSs}

The goal of this section is to prove Theorem~\ref{thm_existence_Birkhoff_sections}.
Throughout this section we consider a broken book decomposition $(K,\F)$ strongly carrying a smooth vector field $X$ on a closed, connected and oriented $3$-manifold $M$  as in Definition~\ref{defn_BB}. The flow of $X$ is denoted by $\phi^t$.

Note that $K_b = \emptyset$ if, and only if, the broken book decomposition is a rational open book decomposition as defined in~\cite{BEvHN}, whose pages are Birkhoff sections.
In this case, let $y \in H^1(M\setminus K;\R)$ be Poincar\'e dual to the class of a page in $H_2(M,K;\Z)$.
Assumption (ii) in Theorem~\ref{thm_existence_from_SFS} follows from (II.a) in Definition~\ref{defn_BB}.
Assumption (i) in Theorem~\ref{thm_existence_from_SFS} follows from the fact that the return time back to a page is uniformly bounded; (II.a) in Definition~\ref{defn_BB} plays an important role here.
Theorem~\ref{thm_existence_from_SFS} provides the desired $\partial$-strong Birkhoff section.

\begin{remark}
The notion of $\partial$-strong transverse surfaces from~\cite{FH} was incorporated in~\cite{CDR}.
It is shown by~\cite[Theorem~1.1]{CDR} that a nondegenerate Reeb vector field on a closed $3$-manifold is strongly carried by a broken book decomposition whose pages are $\partial$-strong sections.
For the purposes of proving Theorem~\ref{thm:BirkhoffSection}, it suffices to prove a version of Theorem~\ref{thm_existence_Birkhoff_sections} where one starts from a vector field strongly carried by a broken book decomposition with this additional property.
Hence, the pages of a broken book are $\partial$-strong Birkhoff sections when $K_b =\emptyset$.
\end{remark}

Hence we proceed assuming that $K_b \neq \emptyset$.

\subsection{Blowing periodic orbits up} 
\label{ssec_blowing_up}

The discussion here makes no use of the broken book decomposition.
 We consider any link $L \subset M$ made of periodic orbits of~$\phi^t$, oriented by the flow. 
For each orbit $\gamma \subset L$ denote by $T_\gamma > 0$ its primitive period, and let $N_\gamma$ be a small tubular neighborhood of~$\gamma$ with an orientation preserving diffeomorphism $\Psi_{\gamma}:N_{\gamma} \to \R/T_{\gamma}\Z\times \D$ as in~\eqref{tubular_map}.
On $N_{\gamma} \setminus \gamma$ we have tubular polar coordinates $(t,r,\theta) \in \R/T_\gamma\Z \times (0,1]\times\R/2\pi\Z$, $\Psi_{\gamma}^{-1}(t,re^{i\theta}) \simeq (t,r,\theta)$ as in~\eqref{polar_tubular_coordinates}.
One can blow the link $L$ up to construct a smooth $3$-manifold $M_L$ defined as
\begin{equation}
\label{blown_up_manifold}
M_L := \left. \left\{ M\setminus L \ \ \sqcup \ \ \bigsqcup_{\gamma \subset L} \R/T_{\gamma}\Z \times (-\infty,1]\times \R/2\pi\Z \right\} \right/ \sim
\end{equation}
where $(t,r,\theta)$ in $\R/T_{\gamma}\Z \times (0,1] \times \R/2\pi\Z$ gets identified with $\Psi^{-1}_\gamma(t,re^{i\theta})$. 
On $M_L$ there are $\#\pi_0(L)$ smoothly embedded tori
\begin{equation}
\label{tori_new}
\Sigma_\gamma = \R/T_{\gamma}\Z \times \{0\} \times \R/2\pi\Z
\end{equation}
with coordinates $(t,\theta) \in \R/T_{\gamma}\Z \times \R/2\pi\Z$, and $M_L$ contains a smooth compact domain
\begin{equation}
\label{domain_D_L}
D_L = \left. \left\{ M\setminus L \ \ \sqcup \ \ \bigsqcup_{\gamma \subset L} \R/T_{\gamma}\Z \times [0,1]\times \R/2\pi\Z \right\} \right/ \sim
\end{equation}
satisfying
$$
\partial D_L = \bigsqcup_{\gamma \subset L} \Sigma_\gamma \, .
$$
Note that $D_L \setminus \partial D_L = M\setminus L$.
Note also that $(t,r,\theta) \mapsto \Psi_{\gamma}^{-1} (t,re^{i\theta})$ defines a smooth diffeomorphism from $\R/T_{\gamma}\Z \times (0,1] \times \R/2\pi\Z$ to $N_{\gamma}\setminus\gamma$.

It follows from arguments originally due to Fried~\cite{fried} that $X$ can be smoothly extended from $M \setminus L$ to a vector field $X_L$ on $M_L$, whose flow we denote by $\phi_L^t$.
The restriction of $X_L$ to $D_L$ is unique, and $X_L$ is tangent to $\partial D_L$.
The details we need on Fried's construction can be found below, for more details see~\cite[Section~3]{Hry_SFS}. 
In particular, it is important to know that there is a precise relation between the dynamics of $X_L$ on $\Sigma_\gamma$ and the linearized dynamics of $X$ along $\gamma$. To see this, denote by $Z = (\Psi_{\gamma})_*X$ the representation of $X$ in $\R/T_{\gamma}\Z \times \D$. By the fundamental theorem of calculus we get that 
\begin{equation*}
Z(t,0) = \begin{pmatrix} 1 \\ 0 \end{pmatrix} \qquad \Rightarrow \qquad Z(t,re^{i\theta}) = \begin{pmatrix} 1 \\ 0 \end{pmatrix} + A(t,re^{i\theta})re^{i\theta}
\end{equation*}
where
\begin{equation*}
A(t,re^{i\theta}) = \int_0^1 D_2Z(t,\tau re^{i\theta}) \ d\tau = \begin{pmatrix} A_1(t,re^{i\theta}) \\ A_2(t,re^{i\theta}) \end{pmatrix}
\end{equation*}
and $D_2$ stands for the partial derivative in the $\D$-factor. 
Note also that
\begin{equation*}
D_2Z(t,0) = \begin{pmatrix} A_1(t,0) \\ A_2(t,0) \end{pmatrix} \, .
\end{equation*}
Now consider the smooth map $$ \Phi : \R/T_{\gamma}\Z \times [0,1] \times \R/2\pi\Z \to \R/T_{\gamma}\Z \times \D \qquad\qquad (t,r,\theta) \mapsto (t,re^{i\theta}) \, . $$
Note that $\Phi$ defines a diffeomorphism $\R/T_{\gamma}\Z \times (0,1] \times \R/2\pi\Z \simeq \R/T_{\gamma}\Z \times (\D\setminus\{0\})$ that can be used to pull $Z|_{\R/T_{\gamma}\Z \times (\D\setminus\{0\})}$ back to a vector field $W$. Since $$ D\Phi^{-1}(t,re^{i\theta}) = \begin{pmatrix} 1 & 0 \\ 0 & \begin{pmatrix} \cos\theta & \sin\theta \\ -r^{-1}\sin\theta & r^{-1}\cos\theta \end{pmatrix} \end{pmatrix} $$ it follows that
\begin{equation*}
\begin{aligned} 
&W(t,r,\theta) = D\Phi^{-1}(t,re^{i\theta}) Z(t,re^{i\theta}) \\
&= (1+A_1(t,re^{i\theta})re^{i\theta}) \partial_t + \left< e^{i\theta},A_2(t,re^{i\theta})re^{i\theta} \right> \partial_r + \left< ie^{i\theta},A_2(t,re^{i\theta})e^{i\theta} \right> \partial_\theta
\end{aligned}
\end{equation*}
extends smoothly to $\R/T_{\gamma}\Z \times [0,1] \times \R/2\pi\Z$. 
Since at $r=0$ the component in $\partial_r$ vanishes, we conclude that the extension $X_L$ of $X$ from $M\setminus L$ to $D_L$ is tangent to~$\partial D_L$.
The restriction of $X_L$ to $\Sigma_\gamma$ is then of the form
\begin{equation}
\label{function_b}
\partial_t + b(t,\theta) \partial_\theta \qquad \qquad b(t,\theta) = \left< ie^{i\theta},A_2(t,0)e^{i\theta} \right> \, .
\end{equation}
With the above formula, the relation between the dynamics of $X_L$ on $\partial D_L$ and the linearized dynamics along the various orbits $\gamma \subset L$ becomes precise, since $b(t,\theta)$ is exactly the same as the function appearing in~\eqref{form_of_ODE_linearized_flow}. A solution $$ \begin{pmatrix} a(t) \\ u(t) \end{pmatrix} = D\phi^t(t_0,0) \begin{pmatrix} a_0 \\ u_0 \end{pmatrix}, \qquad u_0 = |u_0|e^{i\theta_0} \neq 0 $$ along $\gamma$ must satisfy
\begin{equation}
u(t) = |u(t)|e^{i\theta(t)} \qquad\qquad \dot\theta(t) = b(t+t_0,\theta(t)), \quad \theta(t_0)=\theta_0 \, .
\end{equation}

\subsection{Rigid leaves}
\label{ssec_rigid_leaves}

The set of rigid leaves~$R_\F$ was considered in~\cite[Definition~2.8]{CDR}. 
By definition, every leaf of a broken book decomposition is an open surface embedded in $M$.

\begin{definition}
A leaf $L$ of a broken book $(K, \mathcal{F})$ that belongs to the interior of
a 1-parameter family of pages of the form $L\times [0,1]$ is called regular. On the other hand, a leaf
that is not in the interior of such a 1-parameter family is called rigid.
\end{definition}

We have $K_b \neq\emptyset \Rightarrow R_\F \neq\emptyset$, in fact, leaves in the local model of (IV)~Definition~\ref{defn_BB} that divide radial sectors from sectors foliated by hyperbola are rigid.

\begin{lemma}
\label{lemma_foliation_away_rigid_leaves}
Assume that $R_\F \neq\emptyset$.
If $U\subset M$ is a connected open set whose closure does not intersect the union of closures of rigid leaves, then $\F|_U$ is given by the fibers of a submersion $U \to \R$.
\end{lemma}

\begin{proof}
Let $Q$ be the union of the closures of the rigid leaves.
By~\cite[Definition~2.7]{CDR}, each nonrigid leaf belongs  to a local $1$-parameter family of nonrigid leaves. 
Hence, each of the finitely many connected components of $M\setminus Q$ can be foliated by a maximal $1$-parameter family of nonrigid leaves, parametrized by an open interval or a circle.
If one of these maximal families is a circle, then connectedness of $M$ implies that this is the only family, it foliates $M\setminus K$, and there are no rigid leaves. 
Hence, each maximal family is parametrized by an interval.
The parameter of such a family defines a submersion $U\to\R$ whose fibers define $\F|_U$.
\end{proof}


\begin{lemma}
\label{lemma_rigid_at_K_r}
If $K_b\neq \emptyset$ then for every $\gamma \subset K_r$ there is at least one rigid leaf with an end in $\gamma$.
\end{lemma}

\begin{proof}
We have $R_\F \neq\emptyset$ since $K_b \neq\emptyset$. Hence, Lemma~\ref{lemma_foliation_away_rigid_leaves} applies.
As before, let $Q$ be the union of the closures of the rigid leaves.
Assume, by contradiction, that no rigid leaf has an end at a given $\gamma \subset K_r$.
Any neighborhood of $\gamma$ contains a compact subset $E$ with nonempty and connected interior $U$, such that $E \cap Q = \emptyset$ and $U$ contains a loop 
transverse to~$\F|_U$.
Hence, $\F|_U$ cannot be given by the fibers of the map $U \to \R$ given by Lemma~\ref{lemma_foliation_away_rigid_leaves}.
This contradiction concludes the argument.
\end{proof}

Assume $K_b \neq \emptyset$. On each connected open set $U$ whose closure is at a positive distance from the leaves in $R_\F$, the foliation $\F|_U$ is defined by a submersion to~$\R$. 
Such a submersion gives a function that is strictly increasing or strictly decreasing along segments of orbits of $X$ contained in $U$. 
In the arguments that follow we assume, without loss of generality, that such functions are strictly increasing.

\begin{lemma}
\label{lemma_times_hit_rigid}
Fix any metric on $M$. 
For every $\delta>0$ there exists $T_\delta>0$ such that the following holds: if $p \in M \setminus K$ satisfies ${\rm dist}(\phi^{[-T_\delta,T_\delta]}(p),K_b) \geq \delta$, then there exists $t_* \in [-T_\delta,T_\delta]$ such that $\phi^{t_*}(p)$ belongs to some leaf in $R_\F$.
\end{lemma}

\begin{proof}
We give a proof by contradiction. Assume  that there exists  some $\delta>0$ such that the conclusion of the lemma is false. 
There is a sequence of points $p_n\in M \setminus K$ and a sequence $t_n\in \R$ tending to infinity as $n\to \infty$, such that ${\rm dist}(\phi^{[-t_n,t_n]}(p_n),K_b) \geq \delta$ and $\phi^{[-t_n,t_n]}(p_n)\cap (\cup_{\ell\in R_\F}\ell)=\emptyset$. 
Modulo passing to a subsequence,~$p_n$ converges to some point $p_\infty$ such that ${\rm dist}(\phi^{\R}(p_\infty),K_b) \geq \delta$.
Consider the compact set $E = K \cup (\cup_{\ell\in R_\F}\ell)$.

We claim that ${\rm dist}(\phi^{\R}(p_\infty),E)>0$. 
First suppose, by contradiction, that $\phi^{\R}(p_\infty)$ contains points arbitrarily close to $K_r$. 
There exist $s_m\in\R$ and $\gamma \subset K_r$ such that ${\rm dist}(\phi^{s_m}(p_\infty),\gamma) \to 0$ as $m\to+\infty$.
By (II.a) in Definition~\ref{defn_BB}, there is an open neighborhood $W$ of $\gamma$ and some $L>0$ such that $\phi^{[0,L]}(q)$ intersects every leaf with an end at $\gamma$, for every $q\in W \setminus \gamma$.
Fix $m$ such that $\phi^{s_m}(p_\infty) \in W$.
There exists $n_0$ such that $n\geq n_0 \Rightarrow \phi^{s_m}(p_n) \in W\setminus\gamma$.
Hence, by Lemma~\ref{lemma_rigid_at_K_r}, $\phi^{[0,s_m+L]}(p_n)$ intersects some rigid leaf provided $n\geq n_0$. 
This is in contradiction to the existence of $n_1\geq n_0$ such that $t_{n_1}>s_m+L$.
We showed that ${\rm dist}(\phi^{\R}(p_\infty),K_r)>0$, from where it follows that ${\rm dist}(\phi^{\R}(p_\infty),K)>0$.
Again by contradiction, we assume that $\phi^{\R}(p_\infty)$ contains points arbitrarily close to $\cup_{\ell\in R_\F}\ell$. 
Using the transversality of the flow with the leaves and the fact that $\phi^{\R}(p_\infty)$ is at a positive distance to $K$, we find an intersection between $\phi^{\R}(p_\infty)$ and some rigid leaf. 
As before, this implies that $\phi^{[-t_n,t_n]}(p_n)$ intersects some rigid leaf, an absurd.
We are done with the proof of ${\rm dist}(\phi^{\R}(p_\infty),E)>0$.

It follows that the $\omega$-limit set of $p_\infty$ satisfies ${\rm dist}(\omega(p_\infty),E)>0$, hence $\omega(p_\infty)$ is a compact subset of some connected open set $U$ compactly contained in $M\setminus E$.
 With the help of~\cite[Corollary~3.3.7]{HK_book} we can choose a recurrent point $q \in \omega(p_\infty)$.
By Lemma~\ref{lemma_foliation_away_rigid_leaves} there is a continuous function $f:U\to \R$ strictly increasing along the segments of orbits in $U$, in particular along the orbit of~$q$. 
This leads to a contradiction: a continuous function cannot be strictly increasing along a recurrent orbit. 
\end{proof}

\subsection{Invariant measures and intersection numbers}
\label{ssec_invt_measures_int_numbers}

As in the previous section, we consider a link $L\subset M$ tangent to $X$, with the orientation induced by~$X$. We use the notation and the constructions from Section~\ref{ssec_blowing_up}.

Denote by $\P_\phi(M\setminus L)$ the set of $\phi^t$-invariant Borel probability measures on $M\setminus L$, and by $\P_\phi(D_L)$ the set of $\phi^t_L$-invariant Borel probability measures on $D_L$.
If $\mu \in \P_\phi(M\setminus L)$ and $y \in H^1(M\setminus L;\R)$ then we may choose a closed $1$-form $\beta$ on $M\setminus L$ representing $y$ such that $\iota_X\beta$ is bounded, and define
\begin{equation}
\mu \cdot y = \int_{M\setminus L} \iota_X\beta \ d\mu
\end{equation}
called the intersection number of $\mu$ and $y$.
The existence of $\beta$ as above can be seen with the help of polar tubular coordinates, and is implicitly contained in Section~\ref{ssec_blowing_up}, see~\cite[Section~2.1]{Hry_SFS} for more details.
The independence on the choice of $\beta$ can be seen by showing that there exists $f_{\mu,y} \in L^1(\mu)$ and a Borel set $E \subset M\setminus L$ such that:
\begin{itemize}
\item $\mu(E) = 1$ and all points in $E$ are recurrent.
\item If $p \in E$ and $V \subset M\setminus L$ is an open contractible neighborhood of $p$, then $T_n^{-1} \left< y,k(T_n,p) \right> \to f_{\mu,y}(p)$ as $n\to \infty$, where $T_n \to +\infty$ is any sequence such that $\phi^{T_n}(p) \to p$, and $k(T_n,p)$ are choices of loops obtained by concatenating to $\phi^{[0,T_n]}(p)$ a path from $\phi^{T_n}(p)$ to $p$ inside $V$.
\item The identity $$ \mu \cdot y = \int_{M\setminus L} f_{\mu,y} \ d\mu $$ holds.
\end{itemize}
These facts can be proved with a direct application of the ergodic theorem, for more details see~\cite[Section~2.1]{Hry_SFS}.

Similarly, for any $\mu \in \P_\phi(D_L)$ and $y \in H^1(D_L;\R)$ one defines 
\begin{equation}
\mu \cdot y = \int_{D_L} \iota_{X_L}\beta \ d\mu
\end{equation}
where $\beta$ is a closed $1$-form on $D_L$ representing $y$. Independence of $\beta$ is easier in this case since $D_L$ is compact. 

We may freely identify $H^1(D_L;\R) \simeq H^1(M\setminus L;\R)$ via pull-back by the inclusion map $M \setminus L = D_L \setminus \partial D_L \hookrightarrow D_L$, as in the statement below.

\begin{lemma}[\cite{Hry_SFS}, Lemma~3.6]
\label{lemma_rot_numbers_measures}
If $\gamma$ is a component of $L$, $\mu \in \P_\phi(D_L)$ satisfies $\supp(\mu) \subset \Sigma_\gamma$, and $y$ is a class in $H^1(D_L;\R)$, then $\mu \cdot y = \frac{2\pi}{T_\gamma} \rho^y(\gamma)$, where $T_\gamma>0$ is the primitive period of $\gamma$.
\end{lemma}

\subsection{First estimates of intersection numbers}
\label{ssec_int_numbers_ergodic}

The cohomology class 
\begin{equation}
\label{cohom_class_rigid_leaves}
y_0 = \sum_{\ell \in R_\F} \ell^* \quad \in \quad H^1(M \setminus K;\R) \simeq H^1(D_K;\R),
\end{equation}
dual to the rigid leaves,  plays a key role in our arguments.

\begin{lemma}
\label{lemma_crucial_ergodic}
If $\mu \in \P_\phi(M\setminus K)$ is ergodic then $\mu \cdot y_0 > 0$.
\end{lemma}

\begin{proof}
Any $\gamma \subset K_b$ is a hyperbolic orbit, hence isolated as an invariant set. Moreover, by hyperbolicity and condition (IV), we can find an isolating compact neighborhood $N_\gamma$ of $\gamma$ inside any neighborhood of $\gamma$ fixed \textit{a priori}, and some $\eta>0$ depending on~$N_\gamma$, such that the following holds: If $p \in M \setminus (K\cup W^u(\gamma) \cup W^s(\gamma))$ and if $\omega \neq \emptyset$ is a connected component of $\{t\in\R\mid \phi^t(p) \in N_\gamma \}$, then $\omega = [a,b]$ is a compact interval with non-empty interior, $\phi^{[a-\eta,a)\cup(b,b+\eta]}(p) \cap N_\gamma = \emptyset$, and $\exists t_* \in \omega$ such that $\phi^{t_*}(p)$ belongs to some leaf in $R_\F$.

Denote $W^u(K_b) = \bigcup_{\gamma\subset K_b} W^u(\gamma)$ and $W^s(K_b) = \bigcup_{\gamma\subset K_b} W^s(\gamma)$. 
Taking unions of various small $N_\gamma$ as above, we find an isolating compact neighborhood $N_b$ for~$K_b$ inside any neighborhood of $K_b$ fixed \textit{a priori}, and some $\eta > 0$ depending on $N_b$, with the following property:
\begin{itemize}
\item[($*$)] If $p \in M\setminus ( K \cup W^u(K_b) \cup W^s(K_b) )$ and $\omega$ is a non-empty connected component of $\{t\in\R\mid \phi^t(p) \in N_b\}$, then $\omega = [a,b]$ is a compact interval with non-empty interior, $\phi^{[a-\eta,a)\cup(b,b+\eta]}(p) \cap N_b = \emptyset$, and $\exists t_* \in \omega$ such that $\phi^{t_*}(p) \in \bigcup_{\ell \in R_\F} \ell$.
\end{itemize}

Fix $\varphi:M \to [0,1]$ continuous such that $\varphi$ is identically equal to $1$ near $K_b$, and is supported on a neighborhood $U$ of $K_b$ satisfying $\mu(U \setminus K_b) < \frac{1}{2}$. 
We can proceed assuming, without loss of generality, that $N_b \subset \varphi^{-1}(1)$.
Since $\mu$ is ergodic, the following hold simultaneously for $\mu$-almost all points $p \in M \setminus K$:
\begin{equation}
\label{point_p_erg}
\begin{aligned}
& \limsup_{T\to+\infty} \frac{1}{T} \Leb(\{t\in[0,T] \mid \phi^t(p) \in N_b\}) \\ & \qquad \leq \lim_{T\to+\infty} \frac{1}{T} \int_0^T \varphi(\phi^t(p)) \ dt = \int_{M\setminus K} \varphi \ d\mu < \frac{1}{2} \, ,
\end{aligned}
\end{equation}
\begin{equation}
\label{hitting_times_intersection_number}
\mu \cdot y_0 = \lim_{T\to+\infty} \frac{1}{T} \# \{ t\in [0,T] \mid \phi^t(p) \in \ell \ \text{for some} \ \ell\in R_\F \} \, .
\end{equation}
From now on we fix such~$p$, and denote
\begin{equation*}
J = \{ t\in \R \mid \phi^t(p) \in N_b \} \, .
\end{equation*}
If we fix an auxiliary metric on~$M$, and choose $\delta>0$ such that ${\rm dist}(\cdot,K_b) > \delta$ on $M\setminus N_b$, then by Lemma~\ref{lemma_times_hit_rigid} we find $T_*>0$ such that if $\phi^{[c,d]}(p) \subset M \setminus N_b$ and $d-c \geq T_*$ then $\phi^{[c,d]}(p)$ intersects some leaf in $R_\F$.

By ($*$), for every $T>0$ we write $J \cap [0,T]$ as a finite union of $n(T)$ maximal compact intervals, which are all at least $\eta$ apart from each other. 
We list these intervals as $h_1,\dots,h_{n(T)}$ in an increasing fashion: $\sup h_j \leq \inf h_{j+1}$.
It follows that $\inf h_{j+1} - \sup h_j \geq \eta$.
Note that only $h_1$ and $h_{n(T)}$ may not be connected components of $J$.

Consequently, we can write $[0,T] \setminus J$ as a finite union of $m(T)$ maximal intervals which are open relatively to  $[0,T]$, with $|n(T)-m(T)| \leq 1$, all of which have length at least equal to $\eta$.
Denote by $\omega^+_1,\dots,\omega^+_{m_+(T)}$ those intervals with length~$\geq T_*$.
Denote by $\omega^-_1,\dots,\omega^-_{m_-(T)}$ those intervals with length $< T_*$.
We have $m(T) = m_+(T) + m_-(T)$ and $$ [0,T] \setminus J = \omega^+_1 \cup \dots \cup \omega^+_{m_+(T)} \cup \omega^-_1 \cup \dots \cup \omega^-_{m_-(T)} \, . $$

Denote $H(T) = \# \{ t\in [0,T] \mid \phi^t(p) \in \cup_{\ell\in R_\F} \ell \}$, i.e. $H(T)$ is the number of hitting times between $\phi^{[0,+\infty)}(p)$ and the union of the union of leaves in $R_\F$ up to time $T$.
By~\eqref{hitting_times_intersection_number} we have $T^{-1}H(T) \to \mu \cdot y_0$ as $T\to+\infty$.
Our constructions so far imply that there is at least one hitting time on each $\omega^+_1,\dots,\omega^+_{m_+(T)}$ and, by ($*$), at least one hitting time on each $h_2,\dots,h_{n(T)-1}$. In particular, $H(T) \geq n(T)-2$.
We split the remaining arguments in two cases.

\medskip

\noindent \textit{Case 1.} $\liminf_{T\to+\infty} T^{-1}n(T) > 0$

\medskip 

In this case we are done since $$ \mu \cdot y_0 = \lim_{T\to+\infty} \frac{H(T)}{T} \geq \liminf_{T\to+\infty} \frac{n(T)}{T} > 0 \, . $$

\medskip

\noindent \textit{Case 2.} $\liminf_{T\to+\infty} T^{-1}n(T) = 0$

\medskip 

In this case we can find $T_j \to +\infty$  such that $T_j^{-1}n(T_j) \to 0$ as $j\to+\infty$. Note that
\begin{equation}
\label{unity_identity_erg}
\begin{aligned}
1 
&= \frac{\Leb(J \cap [0,T_j])}{T_j} + \frac{\Leb([0,T_j]\setminus J)}{T_j} \\
&= \frac{\Leb(J \cap [0,T_j])}{T_j} + \frac{1}{T_j} \sum_{s=1}^{m_+(T_j)} \Leb (\omega^+_s) + \frac{1}{T_j} \sum_{s=1}^{m_-(T_j)} \Leb (\omega^-_s) \, .
\end{aligned}
\end{equation}
Moreover, if $j\gg1$ then $T_j^{-1}\Leb(J \cap [0,T_j]) < \frac{1}{2}$ by~\eqref{point_p_erg}, and $$ \frac{1}{T_j} \sum_{s=1}^{m_-(T_j)} \Leb (\omega^-_s) \leq T_*\frac{m_-(T_j)}{T_j} \leq T_* \frac{m(T_j)}{T_j} \leq T_* \frac{n(T_j)+1}{T_j} \to 0 \, . $$ 
Plugging in~\eqref{unity_identity_erg} we get 
\begin{equation}
j\gg 1 \qquad \Rightarrow \qquad \frac{1}{T_j} \sum_{s=1}^{m_+(T_j)} \Leb (\omega^+_s) \geq \frac{1}{2} \, .
\end{equation}
Consider $\lambda_j$ the maximal number of intervals of length $T_*$ that fit inside the union of the $\omega^+_s$.
Note that $$ \left| T_*\lambda_j - \sum_{s=1}^{m_+(T_j)} \Leb (\omega^+_s) \right| \leq T_*m_+(T_j) \, . $$
Hence for $j$ large enough we compute
\begin{equation*}
\frac{1}{2} \leq \frac{1}{T_j} \sum_{s=1}^{m_+(T_j)} \Leb (\omega^+_s) \leq \frac{T_*\lambda_j}{T_j} + \frac{T_*m_+(T_j)}{T_j} \, .
\end{equation*}
Since $\frac{m_+(T_j)}{T_j} \to 0$ we get $\liminf_{j\to\infty} \frac{\lambda_j}{T_j} \geq \frac{1}{2T_*}$. 
Hence, 
\begin{equation*}
\frac{H(T_j)}{T_j} \geq \frac{\lambda_j}{T_j} \geq \frac{1}{4T_*}
\end{equation*}
for all $j$ large enough.
Inequality $\mu \cdot y_0 \geq \frac{1}{4T_*} > 0$ follows.
\end{proof}

\subsection{Further estimates of intersection numbers}

Let the broken book decomposition $(K=K_r \sqcup K_b,\F)$, the link of periodic orbits $K' \subset M \setminus K_b$ and the cohomology class $y' \in H^1(M\setminus K';\R)$ be as in the statement of Theorem~\ref{thm_existence_Birkhoff_sections}. 
As in Section~\ref{ssec_blowing_up}, we may blow $K\cup K'$ up to obtain a new manifold $M_{K\cup K'}$ with a smooth compact domain $D_{K\cup K'}$ such that $D_{K\cup K'} \setminus \partial D_{K\cup K'} = M\setminus (K\cup K')$. 
There is a boundary torus $\Sigma_\gamma$~\eqref{tori_new} for each orbit $\gamma \subset  K\cup K'$ and the vector field $X$ extends smoothly from $M\setminus (K\cup K')$ to $D_{K\cup K'}$.
As agreed before, we may freely identify $H^1(M\setminus (K\cup K');\R)$ with $H^1(D_{K\cup K'};\R)$.

We pull $y'$ back to a cohomology class 
\begin{equation}
\label{class_y''}
y'' \in H^1(M\setminus (K\cup K');\R) \simeq H^1(D_{K\cup K'};\R)
\end{equation}
via the inclusion $M\setminus (K\cup K') \hookrightarrow M \setminus K'$. 
Similarly, we pull the class $y_0$ defined in~\eqref{cohom_class_rigid_leaves} back to a cohomology class 
\begin{equation}
\label{class_y''_0}
y''_0 \in H^1(M\setminus (K\cup K');\R) \simeq H^1(D_{K\cup K'};\R)
\end{equation}
via the inclusion $M\setminus (K\cup K') \hookrightarrow M \setminus K$.
We denote
\begin{equation}
\label{broken_boundary}
\partial_bD_{K\cup K'} = \bigcup_{\gamma \subset K_b} \Sigma_\gamma \, .
\end{equation}

\begin{lemma}
\label{lemma_secondary_ergodic}
If $\mu \in \P_\phi(D_{K\cup K'})$ satisfies $\supp(\mu) \subset \partial_bD_{K\cup K'}$ then $\mu \cdot y_0'' = 0$ and $\mu \cdot y'' > 0$.
\end{lemma}

\begin{proof}
In the construction of $D_{K \cup K'}$ explained in Section~\ref{ssec_blowing_up} we fix a diffeomorphism $\Psi_\gamma$ as in~\eqref{tubular_map} from a small tubular neighborhood $N_\gamma$ of $\gamma \subset K \cup K'$ onto $\R/T_\gamma\Z \times \D$, with induced tubular polar coordinates $(t,r,\theta) \in \R/T_\gamma\Z\times [0,1]\times \R/2\pi\Z$. 
These coordinates model a smooth neighborhood in $D_{K \cup K'}$ of the torus component $\Sigma_\gamma = \{r=0\} \subset \partial D_{K \cup K'}$. 
If $\gamma \subset K_b$ and $0 < \epsilon \ll 1$ then, since $K_b \cap K' = \emptyset$, the loop $$ t\in \R/T_\gamma\Z \mapsto (t,\epsilon,0) \in D_{K \cup K'} \setminus \partial D_{K \cup K'} = M \setminus (K \cup K') $$ is homologous to $\gamma$ in $M \setminus K'$, and the loop $\theta\in \R/2\pi\Z \mapsto (0,\epsilon,\theta)$ is homologous to zero in $M \setminus K'$.
Hence, if we write $y'' \equiv pdt+qd\theta$ in $N_\gamma \setminus \gamma$, then $$ p = \frac{\left< y'', [t \mapsto (t,\epsilon,0)] \right>}{T_\gamma} = \frac{\left< y',\gamma \right>}{T_\gamma} > 0 $$ and $$ q = \frac{\left< y'', [\theta \mapsto (0,\epsilon,\theta)] \right>}{2\pi} = \frac{\left< y', [\theta \mapsto (0,\epsilon,\theta)] \right>}{2\pi} = 0 \, . $$ 
Identifying $H^1(D_{K\cup K'};\R) \simeq H^1(M\setminus (K\cup K');\R)$, we can apply Lemma~\ref{lemma_rot_numbers_measures} to get $$ \mu \cdot y'' =  \frac{2\pi}{T_\gamma} \rho^{y''}(\gamma) = \frac{\left<y',\gamma\right>}{T_\gamma} > 0 \, . $$

We now prove that $\mu \cdot y_0'' = 0$. 
Fix $\ell \in R_\F$ and its dual $\ell^* \in H^1(M\setminus K;\R)$. Pulling $\ell^*$ back to $H^1(M\setminus (K\cup K');\R)$ via the inclusion $M\setminus (K\cup K') \hookrightarrow M\setminus K$, we get a class $y_{\ell}'' \in H^1(M\setminus (K\cup K');\R) \simeq H^1(D_{K\cup K'};\R)$. 
In $N_\gamma \setminus \gamma$, with $\gamma \subset K_b$, we have $y_{\ell}'' \equiv p dt+q d\theta \equiv \ell^*$ for constants $p,q\in\R$, as above.
If $\ell$ does not contain~$\gamma$ in its boundary then $p=q=0$, and $\mu\cdot y_{\ell}'' =0$. 
If $\ell$  contains $\gamma$ in its boundary then $\mu\cdot y_{\ell}'' = \frac{2\pi}{T_\gamma}  \rho^{\ell^*}(\gamma) = 0$ by (II.b) in Definition~\ref{defn_BB}.
Here Lemma~\ref{lemma_rot_numbers_measures} was used.
Since $y_0'' = \sum_{\ell \in R_\F} y''_\ell$ the desired conclusion follows.

\end{proof}

\subsection{Obtaining the Birkhoff section}

The space $C^0(D_{K\cup K'})$ becomes a Banach space with the sup norm. 
Its topological dual $C^0(D_{K\cup K'})'$ is a Banach space with the corresponding dual norm, and its unit ball is a compact metrizable space when equipped with the weak* topology; here it was used $C^0(D_{K\cup K'})$ is a separable Banach space.
The set $\P_\phi(D_{K\cup K'})$ can be seen as a convex subset of the unit ball of $C^0(D_{K\cup K'})'$, and the Riesz representation theorem implies that $\P_\phi(D_{K\cup K'})$ is closed in weak*. 
Hence $\P_\phi(D_{K\cup K'})$ with the weak* topology is a compact metric space.

Consider $\E \subset \P_\phi(D_{K\cup K'})$ the subset of ergodic measures.
Let $\E_b$ be the set of those $\mu \in \E$ satisfying $\mu(\partial_bD_{K\cup K'})=1$.
Here $\partial_bD_{K\cup K'}$ is the invariant set~\eqref{broken_boundary}. 
Then $\E_1 = \E \setminus \E_b$ is the set of those $\mu \in \E$ satisfying $\mu(\partial_bD_{K\cup K'})=0$.
In general, the sets $\E,\E_b,\E_1$ might be non-compact in weak*.

Every cohomology class $\alpha \in H^1(D_{K\cup K'};\R)$ defines a function $$ \mu \in \P_\phi(D_{K\cup K'}) \mapsto \mu \cdot \alpha \in \R $$ which is weak* continuous.
This is so since, by definition of intersection numbers, $\mu \cdot \alpha$ is an integral of some function in $C^0(D_{K\cup K'})$ with respect to $\mu$.
Consider the classes $y''$ and $y''_0$ defined in~\eqref{class_y''} and in~\eqref{class_y''_0}, respectively.

Let $\mu \in \E_1$. By ergodicity, either $\mu(M\setminus (K\cup K')) = 1$ or $\mu(M\setminus (K\cup K')) = 0$. 
In the former case, $\mu$ restricts to $M\setminus (K\cup K')$ as an ergodic measure for the flow $\phi^t$ on $M\setminus (K\cup K')$, and Lemma~\ref{lemma_crucial_ergodic} implies that $\mu \cdot y_0'' > 0$. 
 In the latter case, we find a component $\gamma \subset K_r \cup K'$ such that $\supp(\mu)$ is contained on the boundary torus~$\Sigma_\gamma$.
If $\gamma \subset K_r$ then Lemma~\ref{lemma_rot_numbers_measures} implies $\mu \cdot y_0'' > 0$. Here (II.a) from Definition~\ref{defn_BB} is also used. 
If $\gamma \subset K' \setminus K_r$ then we get $\mu \cdot y_0'' > 0$ from the fact that $\gamma$ is a periodic orbit in the complement of $K$, as such it has to intersect some rigid page and every intersection point is transverse and positive. 
Summarizing, we proved that
\begin{equation}
\label{pos_E_1_y''_0}
\mu \in \E_1 \qquad \Rightarrow \qquad \mu \cdot y_0'' > 0 \, .
\end{equation}

By compactness of $\P_\phi(D_{K\cup K'})$ we get $$ c = \sup_{\mu \in \P_\phi(D_{K\cup K'})} |\mu \cdot y''| < + \infty \, . $$
Consider
\begin{equation}
Z = \overline{\E_1} \cap \{ \mu \in \P_\phi(D_{K\cup K'}) \mid \mu \cdot y''_0 = 0 \} 
\end{equation}
where $\overline{\E_1}$ denotes the weak* closure.
Note that $Z$ is compact in weak* since it is a closed subset of the compact set $\P_\phi(D_{K\cup K'})$.

\begin{lemma}
\label{lemma_positivity_in_Z}
If $\mu \in Z$ then $\mu \cdot y''>0$.
\end{lemma}

\begin{proof}
Fix $\mu \in Z$. By the ergodic decomposition theorem~\cite[page 84]{HK_handbook} we can find a Borel probability measure $P_\mu$ on $\E$ such that
\begin{equation}
\label{erg_decomp_thm}
\mu \cdot \alpha = \int_\E \nu \cdot \alpha \ dP_\mu(\nu)
\end{equation}
holds for every $\alpha \in H^1(D_{K\cup K'};\R)$. 
Apply this formula to $y_0''$ in combination with the definition of $Z$ and with Lemma~\ref{lemma_secondary_ergodic} to get
\begin{equation*}
\begin{aligned}
0 = \mu \cdot y_0'' = \int_\E \nu \cdot y_0'' \ dP_\mu(\nu) 
&= \int_{\E_b} \nu \cdot y_0'' \ dP_\mu(\nu)  + \int_{\E_1} \nu \cdot y_0'' \ dP_\mu(\nu) \\
&= \int_{\E_1} \nu \cdot y_0'' \ dP_\mu(\nu) \, .
\end{aligned}
\end{equation*}
Now~\eqref{pos_E_1_y''_0} implies that $\nu \cdot y_0''>0$ for all $\nu \in \E_1$, and we conclude that $P_\mu(\E_1)=0$.
Hence $1=P_\mu(\E_b)$.
Substituting $\alpha = y''$ in~\eqref{erg_decomp_thm} we finally get $$ \mu \cdot y'' = \int_{\E_b} \nu \cdot y'' \ dP_\mu(\nu) > 0 $$ since by Lemma~\ref{lemma_secondary_ergodic} the integrand is pointwise strictly positive.
\end{proof}

By compactness of $Z$ and Lemma~\ref{lemma_positivity_in_Z}, we find an open neighborhood $U$ of $Z$ in $\P_\phi(D_{K\cup K'})$ such that $\mu \cdot y'' > 0$ for all $\mu \in U$. 
Moreover, since~\eqref{pos_E_1_y''_0} implies that $\mu \cdot y_0'' \geq 0$ for every $\mu \in \overline{\E_1}$, we get $\mu \cdot y_0''> 0$ for every $\mu$ in $Z' = \overline{\E_1} \setminus U$.
By the compactness of $Z'$ $$ d = \inf_{\mu\in Z'} \mu \cdot y_0'' > 0 \, . $$
Finally, choose $a>0$ large enough so that $ad > c$ and compute:
\begin{itemize}
\item If $\mu \in \E_b$ then $\mu \cdot (ay_0''+y'') = \mu \cdot y'' > 0$.
\item Assume that $\mu \in \E_1$. If $\mu \in U$ then $\mu \cdot (ay_0''+y'') = a \mu \cdot y_0''  + \mu \cdot y''$ and both $\mu \cdot y_0''$ and $\mu \cdot y''$ are strictly positive. If $\mu \in \E_1 \setminus U \subset Z'$ then $\mu \cdot (ay_0''+y'') \geq a d  -c > 0$.
\end{itemize}
This proves that the class $y_*'' = ay_0''+y'' \in H^1(D_{K\cup K'};\R)$ satisfies $\mu \cdot y_*'' > 0$ for all $\mu \in \E$. The ergodic decomposition theorem implies that $\mu \cdot y_*'' > 0$ must also hold for all $\mu \in \P_\phi(D_{K\cup K'})$. Let $y_* \in H^1(M\setminus(K\cup K');\R)$ be the class obtained by pulling $y_*''$ back via the inclusion $M\setminus(K\cup K') = D_{K\cup K'} \setminus \partial D_{K\cup K'} \hookrightarrow D_{K\cup K'}$. Then $\mu \cdot y_*>0$ holds for all $\mu \in \P_\phi(M\setminus(K\cup K'))$, and Lemma~\ref{lemma_rot_numbers_measures} implies that $\rho^{y_*}(\gamma)>0$ for all $\gamma \subset K \cup K'$. Hence we can apply Theorem~\ref{thm_existence_from_SFS} to find the desired $\partial$-strong Birkhoff section.

\section{Finding orbits that link the broken binding}
\label{sec_general_scheme}

Our goal here is to prove Proposition~\ref{prop_gen_1}.
Let $\lambda$ be a contact form on a closed, connected and oriented $3$-manifold~$M$ satisfying $\lambda \wedge d\lambda > 0$, with Reeb vector field~$X$.
We consider a sequence of weighted links of periodic Reeb orbits $(\{\gamma^n_j\},\{p^n_j\})$, $1\leq j \leq N(n)$, as in the introduction, and the associated Borel invariant probability measures $$ \mu^n = \sum_j p^n_j \frac{(\gamma^n_j)_*\Leb}{T(\gamma^n_j)} \, . $$ 
In Proposition~\ref{prop_gen_1} it is assumed that $(\{\gamma^n_j\},\{p^n_j\})$ as above can be found so that
\begin{equation}
\mu^n \to \frac{\lambda \wedge d\lambda}{{\rm vol}(\lambda)}
\end{equation}
in the weak* topology of $C^0(M)'$.
If we denote $I^n = \gamma^n_1 \cup \dots \cup \gamma^n_{N(n)}$, then there is no loss of generality to assume that
\begin{equation}
\label{links_are_increasing}
I^n \subset I^{n+1} \qquad \qquad \forall n \, .
\end{equation}
Moreover, since $L$ has measure zero with respect to $\lambda \wedge d\lambda$, there is no loss of generality to further assume that 
\begin{equation}
I^n \subset M \setminus L \qquad \qquad \forall n \, .
\end{equation}

For every open set $U \subset M$ and $p\geq0$, $\Omega^p(U)$ is equipped with the $C^\infty_{\rm loc}$-topology, and the space $C_p(U)$ of $p$-currents with compact support in $U$ is the topological dual of $\Omega^p(U)$ with the associated weak* topology.
Boundary operators $\partial_{*}:C_{*}(U)\to C_{*-1}(U)$ are defined as adjoints of exterior derivatives. 
Elements in $Z_p(U) = \ker \partial_p$ are called cycles, and the elements of $B_p(U) = {\rm im} \ \partial_{p+1} \subset Z_p(M)$ are called boundaries. 
The quotient space $Z_p(U)/B_p(U)$ is canonically isomorphic to $H_p(U;\R)$.

Let us enumerate the components of $L$ as $h_1,\dots,h_m$. 
They can also be seen as maps $h_k:\R/T(h_k)\Z \to M$ parametrized by the flow, where $T(h_k)>0$ is the primitive period of $h_k$. 
In this way, each $h_k$ defines a cycle in $C_1(M)$ by the formula $\omega \in \Omega^1(M) \mapsto \int_{\R/T(h_k)\Z}h_k^*\omega$. This cycle is also denoted by $h_k$, with no fear of ambiguity.

\begin{lemma}
\label{lemma_action_linking}
Let $x_1,\dots,x_m \in \Z$ be such that $c = \sum_{k=1}^m x_kh_k \in B_1(M)$. 
Choose an oriented rational Seifert surface $f:S\to M$ for $c$, i.e. $S$ is an oriented compact surface, $f$ is an immersion, and 
\begin{itemize}
\item $f|_{\partial S}$ defines a submersion $\partial S \to L$ covering each $h_k$ exactly $x_k$ times,
\item $f|_{S\setminus \partial S}$ defines an embedding $S\setminus \partial S \to M\setminus L$.
\end{itemize}
Then 
$$ 
\lim_{n\to\infty} \sum_j p^n_j \frac{\vol(\lambda) \ {\rm int}(\gamma^n_j,f)}{T(\gamma^n_j)} = \int_S f^*d\lambda = \sum_{k=1}^m x_kT(h_k) 
$$ 
where ${\rm int}(\gamma^n_j,f) \in \Z$ denotes the algebraic intersection number between $\gamma^n_j$ and $f$.
\end{lemma}

\begin{proof}
Consider the $3$-manifold $M_L$ obtained from $M$ and $L$ via the blow-up construction explained in Section~\ref{ssec_blowing_up}.
It contains a special smooth domain $D_L \subset M_L$~\eqref{domain_D_L} such that $D_L \setminus \partial D_L = M\setminus L$. 
The vector field $X$ extends smoothly from $M\setminus L$ to $D_L$ tangentially to $\partial D_L$.
The extended vector field is still denoted here by $X$, for simplicity.
It is not difficult to show that also $\lambda$ extends smoothly to $D_L$.
Unfortunately, $\lambda$ does not define a contact form on $D_L$ since $d\lambda$ vanishes on $T_y D_L$ for every $y \in \partial D_L$.
The rest of this proof is a small modification of the proof of the Action-Linking Lemma~\cite[Lemma~1.12]{BSHS}.

Consider the map $\pi : D_L \to M$ defined by collapsing each boundary torus~\eqref{tori_new} $\Sigma_{h_k}$ to $h_k$.
This map is smooth since in tubular polar coordinates $(t,r,\theta)$ near $\Sigma_{h_k} = \{r=0\}$, as in~\eqref{polar_tubular_coordinates}, the map is represented by $(t,r,\theta)\mapsto (t,re^{i\theta})$. 
After a $C^\infty$-small perturbation of $f$, we can assume that $f= \pi|_S$ for an embedded surface $S \subset D_L$ that intersects $\partial D_L$ transversely; in particular, $\partial S = S \cap \partial D_L$.

Consider a compact neighborhood $U$ of $S$ diffeomorphic to $[-1,1]\times S$ in such a way that $\{0\} \times S \simeq S$.
Denote by $z$ the coordinate on $[-1,1]$.
If $\Omega$ is a positive area form on $S$ then there is no loss of generality to assume that $\Omega \wedge dz > 0$ on~$U$.
Fix any non-negative test function $\varphi$ on $\R$ such that $\supp(\varphi) \subset [-1,1]$ and $\int_\R\varphi=1$. 
For every $\delta \in (0,1)$ set $\varphi_\delta(x) = \delta^{-1}\varphi(\delta^{-1}x)$.
Hence $\supp(\varphi_\delta) \subset [-\delta,\delta]$ and $\int_\R\varphi_\delta=1$. Note that $\varphi_\delta(z)dz$ defines a closed $1$-form supported in the interior of $U$ and, as such, it can be smoothly extended as a closed $1$-form $\beta_\delta$ on $D_L$. 
It represents the cohomology class dual to $S$.
Consider smooth functions $f,g:U \to \R$ defined by $f = i_Xdz$ and $\lambda \wedge d\lambda = g \ \Omega \wedge dz$.
Then
\begin{equation*}
\begin{aligned}
& \int_{D_L} i_X\beta_\delta \ \lambda \wedge d\lambda = \int_{U \simeq [-1,1] \times S} \varphi_\delta(z) f(z,q) g(z,q) \ \Omega \wedge dz \\
& \qquad = \int_{U \simeq [-1,1] \times S} \varphi_\delta(z) (f(0,q) g(0,q) + \epsilon(z,q)) \ \Omega \wedge dz \\
& \qquad = \int_{S} f(q,0) g(q,0) \ \Omega \ + \ \int_{U \simeq [-1,1] \times S} \varphi_\delta(z) \epsilon(z,q) \ \Omega \wedge dz \, .
\end{aligned}
\end{equation*}
One finds $c>0$ such that $\sup_{q\in S}|\epsilon(z,q)| \leq c|z|$ holds for all $(z,q) \in [-1,1] \times S \simeq U$.
This follows from the fundamental theorem of calculus.
Hence
\begin{equation*}
\begin{aligned}
& \left| \int_{D_L} i_X\beta_\delta \ \lambda \wedge d\lambda - \int_{S} f(0,q) g(0,q) \ \Omega \right| \\
& \qquad \leq c\delta \left( \int_{S} \Omega \right) \left( \int_{[-1,1]} \varphi_\delta(z) dz \right) = c\delta \int_{S} \Omega \, .
\end{aligned}
\end{equation*}
Since all $\beta_\delta$'s are cohomologous, the integral $\int_{D_L} i_X\beta_\delta \ \lambda \wedge d\lambda$ does not depend on~$\delta$.
Taking the limit as $\delta\to0$
\begin{equation}
\label{first_identity_integral_lemma}
\int_{D_L} i_X\beta_{\delta_0} \ \lambda \wedge d\lambda = \int_{S} fg \ \Omega \qquad \forall \delta_0 \in (0,1) \, .
\end{equation}
The identity $d\lambda = i_X(\lambda \wedge d\lambda)$ is valid on $D_L \setminus \partial D_L = M \setminus L$, and hence is also valid on $D_L$ with the smooth extensions of $X$, $\lambda$ and $d\lambda$.
On $U$ one computes 
$$ 
d\lambda = i_X(\lambda \wedge d\lambda) = i_X(g \ \Omega \wedge dz) = fg \ \Omega + \nu \wedge dz 
$$ 
for some $1$-form $\nu$. 
Since $dz$ vanishes tangentially to $S$ we get from~\eqref{first_identity_integral_lemma} that 
$$ 
\int_{D_L} i_X\beta_\delta \ \lambda \wedge d\lambda = \int_S d\lambda = \sum_{k=1}^m x_kT(h_k) \qquad \forall \delta \in (0,1) \, . 
$$
Finally, once $\delta \in (0,1)$ is fixed arbitrarily, the function $i_X\beta_\delta$ is bounded on $M\setminus L$ since it is continuous on the compact space $D_L \supset D_L \setminus \partial D_L = M\setminus L$. Hence
\begin{equation*}
\begin{aligned}
\int_{D_L} i_X\beta_\delta \ \lambda \wedge d\lambda &= \int_{M \setminus L} i_X\beta_\delta \ \lambda \wedge d\lambda \\
&= \vol(\lambda) \lim_{n\to\infty} \int_{M \setminus L} i_X\beta_\delta \ d\mu^n \\
&= \vol(\lambda) \lim_{n\to\infty} \sum_j p^n_j \frac{{\rm int}(\gamma^n_j,f)}{T(\gamma^n_j)} \, .
\end{aligned}
\end{equation*}
Here we used the assumption that $\mu^n \to (\lambda \wedge d\lambda)/\vol(\lambda)$ in weak* topology of $C^0(M)'$ in combination with Lemma~\ref{lemma_measures_conv}.
\end{proof}

Let $H = \span\{h_1,\dots,h_m\} \subset Z_1(M)$. 
Then $\dim H = m$ and $\{h_1,\dots,h_m\}$ is a basis of $H$.
The space $H$ contains a compact convex set
\begin{equation*}
\Cc = \left\{ \left. \sum_{k=1}^m a_kh_k \ \right| \ a_k \geq 0 \ \forall k, \quad \sum_{k=1}^m a_k=1 \right\} \, .
\end{equation*}
Choose a basis $e_1,\dots,e_R$ of $H \cap B_1(M)$. 
Hence $[e_r] = 0$ in $H_1(M;\R)$. 

The universal coefficients theorem tells us that $H_2(M,L;\R) \simeq H_2(M,L;\Q) \otimes \R$.
Hence, for every~$r$ we can find finitely many $y^s_r \in \R$ and oriented rational Seifert surfaces $f^s_r$ as in Lemma~\ref{lemma_action_linking} such that 
$$
e_r = \partial \sum_sy^s_rf^s_r
$$
where the $f^s_r$ are seen in $C_2(M)$.

\begin{lemma}
\label{lemma_asymptotic_positivity}
If $n$ is large enough then $\Cc \cap B_1(M \setminus I^n) = \emptyset$.
\end{lemma}

\begin{proof}
There is no loss of generality to assume that $\vol(\lambda)=1$.
On the space $H$ we have the norm $\|\sum_{k=1}^m a_kh_k\| = \sum_{k=1}^m |a_k|$, and on $H \cap B_1(M)$ we have the norm $\|\sum_{r=1}^R b_re_r\|_0 = \sum_{r=1}^R |b_r|$. 
The norm on $H \cap B_1(M)$ obtained by restricting $\|\cdot\|$ to $H \cap B_1(M)$ is denoted by $\|\cdot\|_1$. 
Since these are norms on a finite-dimensional vector space, we find $\kappa>0$ such that $\|\cdot\|_0 \leq \kappa \|\cdot\|_1$.
Choose $\eta > 0$ satisfying $\eta \geq \max_{r,s}|y^s_r|$.
It is important to note that both $\kappa$ and $\eta$ are independent of $n$. 
By Lemma~\ref{lemma_action_linking}
\begin{equation*}
n\gg1 \qquad \Rightarrow \qquad \left|\sum_j p^n_j \frac{{\rm int}(\gamma^n_j,f^s_r)}{T(\gamma^n_j)} - \left< f^s_r,d\lambda \right> \right| < \frac{1}{2\kappa\eta} \ \min_k T(h_k) \quad \forall r \, .
\end{equation*}
If $n$ is large enough and $c\in \Cc \cap B_1(M)$, write $c = \sum_{k=1}^m a_k h_k = \sum_{r=1}^R b_re_r$ with $\sum_{k=1}^m a_k = 1$ and estimate
\begin{equation*}
\begin{aligned}
& \left| \sum_{j,r,s} b_ry^s_r p^n_j \frac{{\rm int}(\gamma^n_j,f^s_r)}{T(\gamma^n_j)} - \sum_{k=1}^m a_k T(h_k) \right| 
= \left| \sum_{r,s} b_ry^s_r \left( \sum_j p^n_j \frac{{\rm int}(\gamma^n_j,f_r)}{T(\gamma^n_j)} - \left< f^s_r,d\lambda \right> \right) \right|  \\
& \qquad \leq \left( \frac{1}{2\kappa\eta} \min_k T(h_k) \max_{r,s}|y^s_r| \right) \sum_{r=1}^R |b_r| 
\leq \left( \frac{1}{2\kappa} \min_k T(h_k) \right) \|c\|_0 \\
& \qquad \leq \left( \frac{1}{2} \min_k T(h_k) \right) \|c\|_1 = \left( \frac{1}{2} \min_k T(h_k) \right) \sum_{k=1}^m a_k = \frac{1}{2} \min_k T(h_k) \, .
\end{aligned}
\end{equation*}
Combining with 
$$
\sum_{k=1}^m a_k T(h_k) \geq (\min_k T(h_k)) \sum_{k=1}^m a_k = \min_k T(h_k)
$$
and using the triangle inequality, we conclude that 
\begin{equation}
\label{estimate_intersections_1}
n \gg 1 \qquad \Rightarrow \qquad \sum_{j,r,s} b_ry^s_r p^n_j \frac{{\rm int}(\gamma^n_j,f^s_r)}{T(\gamma^n_j)} \geq  \frac{1}{2} \min_k T(h_k) \, .
\end{equation}

Let $\Theta_n$ be a closed $2$-form on $M$ representing the Poincar\'e dual in $H^2(M;\R)$ of the class in $H_1(M;\R)$ represented by the cycle $\Sigma_j \ p^n_j T(\gamma^n_j)^{-1} \ \gamma^n_j$.
Observe that for any fixed closed $1$-form $\alpha$ on $M$ we can compute
$$
\begin{aligned}
\lim_{n\to\infty} \int_M\alpha \wedge \Theta_n 
&= \lim_{n\to\infty} \sum_j \frac{p^n_j}{T(\gamma^n_j)} \int_{\gamma^n_j}\alpha \\
&= \lim_{n\to\infty} \int_M i_X\alpha \ d\mu^n \\
&= \frac{1}{\vol(\lambda)} \int_M i_X\alpha \ \lambda \wedge d\lambda \\
&= \frac{1}{\vol(\lambda)} \int_M \alpha \wedge d\lambda = 0 \, .
\end{aligned}
$$
In particular, this can be applied to the closed $1$-form $\alpha$ representing the Poincar\'e dual of the class in $H_2(M;\R)$ induced by an arbitrary $S \in Z_2(M)$, thus giving
\begin{equation*}
\lim_{n\to\infty} \left<S,\Theta_n\right> = 0 \, .
\end{equation*}

To every $A \in C_2(M)$ satisfying $\partial A \in H$ we can linearly associate $\hat A$ defined by 
$$ \partial A = \sum_rb_re_r \qquad \Rightarrow \qquad \hat A = \sum_{r,s} b_ry^s_rf^s_r \, . $$
Here we used $\partial A \in H \cap B_1(M)$.
In particular, $\partial \hat A = \sum_rb_re_r$ and $\partial (A-\hat A)=0$.

To prove the lemma we now argue by contradiction, and assume that we can find $n_j \to \infty$ and $c_j \in \Cc \cap B_1(M\setminus I^{n_j})$.
Let $A_j \in C_2(M\setminus I^{n_j})$ satisfy $\partial A_j = c_j$.
By what was proved above, we can select further subsequences and assume, with no loss of generality, that 
\begin{equation*}
\left|\left< A_j - \hat A_j,\Theta_{n_{j+1}} \right>\right| < \frac{1}{4} \min_kT(h_k) \qquad \left< \hat A_j,\Theta_{n_{j+1}} \right> \geq \frac{1}{2} \min_kT(h_k) \, .
\end{equation*}
It follows that
\begin{equation}
\label{ineq_A_j}
\left< A_j,\Theta_{n_{j+1}} \right> > \frac{1}{4} \min_kT(h_k) \qquad \forall j \, .
\end{equation}

If $m\geq l$ then $\left<A_m,\Theta_{n_l}\right>=0$ since $A_m$ is supported in $M \setminus I^{n_m} \subset M \setminus I^{n_l}$. Here we used~\eqref{links_are_increasing}.
Each $A_j$ induces a class $[A_j] \in H_2(M,L;\R)$. 
Since the dimension of $H_2(M,L;\R)$ is finite, we can find $j_*$ such that all $[A_j]$ are in the span of $\{[A_1],[A_2],\dots,[A_{j_*}]\}$. 
Write $[A_j] = d_{j1}[A_1] + \dots + d_{jj_*}[A_{j_*}]$. 
We will now argue to prove that $j>j_* \Rightarrow [A_j]=0$. 
Consider $j>j_*$. Then $j\geq 2$ and with the help of~\eqref{ineq_A_j} one gets
$$
0 = \left<A_j,\Theta_{n_2}\right> = d_{j1} \left<A_1,\Theta_{n_2}\right> \Rightarrow d_{j1} = 0 \, .
$$
We are done if $j_*=1$.
If $j_*>1$ then $j\geq 3$ and $0 = \left<A_j,\Theta_{n_3}\right> = d_{j2} \left<A_2,\Theta_{n_3}\right>$ implies that $d_{j2} = 0$.
We are done if $j_*=2$, and if not then $j\geq 4$ and one argues similarly to conclude that $d_{j3}=0$, and so on. This inductive process shows that $d_{js} = 0$ for every $1\leq s\leq j_*$.
This proves that $[A_j]$ vanishes in $H_2(M,L;\R)$ when~$j$ is large enough, in contradiction to $\left<A_j,\Theta_{n_{j+1}}\right>>0$.
\end{proof}

To conclude the proof of Proposition~\ref{prop_gen_1}, fix $n$ so large that $\Cc \cap B_1(M\setminus I^n) = \emptyset$. 
This is possible by the previous lemma.
We invoke the Hahn-Banach theorem to obtain a continuous linear functional $\varphi_n: C_1(M\setminus I^n) \to \R$ satisfying
\begin{equation}
\label{extension_positive_on_the_cone}
\varphi_n|_{\Cc} > 0 \qquad \text{and} \qquad \varphi_n|_{B_1(M\setminus I^n)} \equiv 0 \, .
\end{equation}
By reflexivity~\cite[\S 17]{derham}, we find $\omega_n \in \Omega^1(M\setminus I^n)$ such that $\varphi_n(\cdot) = \left< \cdot,\omega_n \right>$. 
Since~$\varphi_n$ vanishes on $B_1(M\setminus I^n)$ we get $d\omega_n=0$. 
Since all $h_k$ belong to $\Cc$, we get from~\eqref{extension_positive_on_the_cone} that 
\begin{equation*}
\int_{h_k} \omega_n > 0 \qquad \forall k=1,\dots,m \, .
\end{equation*}
The cohomology class of $\omega_n$ in $H^1(M\setminus I^n;\R)$ is the class we were looking for, and the proof of Proposition~\ref{prop_gen_1} is complete.

\section{Openness of the supporting condition}\label{section: open}

\begin{proposition}
\label{prop: open}
Let $X$ be a smooth vector field on a smooth closed $3$-manifold~$M$ that has a $\partial$-strong Birkhoff section $\iota:S \to M$, such that $\iota(\partial S)$ consists of nondegenerate periodic orbits.
There exists a $C^1$-neighborhood $\mathcal{N}(X)$ of $X$ such that every smooth vector field $X' \in \mathcal{N}(X)$ has a $\partial$-strong Birkhoff section isotopic to $\iota$.
\end{proposition}

\begin{proof}
Denote $K = \iota(\partial S)$, $\ell = \iota(S\setminus \partial S)$, and let $\ell^* \in H^1(M\setminus K;\R)$ be the dual class.
As in Section~\ref{ssec_blowing_up}, we blow $M$ up along $K$ to get the compact manifold with boundary $D_K$.
There is a smooth projection $P:D_K \to M$ represented as $(t,r,\theta) \mapsto (t,re^{i\theta})$ with the aid of appropriate polar tubular coordinates~\eqref{polar_tubular_coordinates}.
The map $P$ defines a diffeomorphism $D_K\setminus \partial D_K \to M\setminus K$ and collapses $\Sigma_\gamma$ to~$\gamma$.
The vector field $X\vert_{M\setminus K}$ extends smoothly to $D_K$ as a nonsingular vector field tangent to $\partial D_K$. 
The extended vector field is denoted by $\tilde X$.
In fact, there is a boundary torus $\Sigma_\gamma \subset \partial D_K$ associated to every orbit $\gamma \subset K$, and the tubular polar coordinates $(t,\theta)$ are global periodic coordinates on $\Sigma_\gamma \simeq \mathbb{P}^+\gamma$ that represent the projection to~$\gamma$ as $(t,\theta) \mapsto t$.
In these coordinates $\tilde X$ assumes the form $\partial_t+b(t,\theta)\partial_\theta$ and its dynamics is precisely linearized dynamics on $\mathbb{P}^+\gamma$.
It is straightforward to construct a smooth map $\tilde\iota:S \to D_K$ satisfying $\iota = P \circ \tilde\iota$.
The assumption that $\iota$ is $\partial$-strong implies that, up to a $C^\infty$-small perturbation, $\tilde\iota(S)$ is an embedded surface in $D_K$ transverse to $\partial D_K$. 
Any vector field $X'$ on $M$ that coincides with $X$ on $K$ extends smoothly to a vector field $\tilde X'$ on $D_K$, tangent to $\partial D_K$.
Moreover, if $X'$ is $C^1$-close to $X$ then $\tilde X'$ is $C^0$-close to $\tilde X$. This follows from the construction in Section~\ref{ssec_blowing_up}. In fact, referring back to the notation established in Section~\ref{ssec_blowing_up}, if $A_1(1,re^{i\theta}),A_2(t,re^{i\theta})$ and $A_1'(t,re^{i\theta}),A_2'(t,re^{i\theta})$ are the functions corresponding to $X$ and $X'$, respectively, then the $C^k$-norm of $A_j-A_j'$ on a compact neighborhood of some component of $K$ is controlled by the $C^{k+1}$-norm of $X-X'$. 
It follows that $\tilde\iota(S)$ is transverse to $\tilde X'$ up to $\partial D_K$ and there is a smooth (hence bounded) return time to $\tilde\iota(S)$ for the dynamics of $\tilde X'$.
It follows that $\iota$ is a $\partial$-strong Birkhoff section for $X'$ provided that $X'$ is $C^1$-close enough to~$X$.

Now take a transverse disk section $D_\gamma$ near a component $\gamma$ of $K$, where $\gamma$ appears as an isolated fixed point of the first return map.
By the implicit function theorem, this fixed point varies slightly under $C^1$-small perturbations of~$X$ and defines a periodic orbit $\gamma'$ for the perturbed vector field, $C^1$-close to $\gamma$. 
Combined with the ``extension of isotopies'' theorem, we obtain that for every~$\epsilon>0$ there exists a $C^1$-small neighborhood $\mathcal{N}_\epsilon(X)$ of $X$ such that for every smooth vector field $X'\in \mathcal{N}_\epsilon (X)$ there is an $\epsilon$-small (in $C^1$) smooth isotopy $(\phi_t)_{t\in [0,1]}$, $\phi_0=id_M$, of $M$ taking $\gamma$ to a periodic orbit $\gamma'=\phi_1(\gamma)$ of $X'$, and similarly for the other binding components.
This isotopy can be taken supported near $K$.
By further multiplying by a function $h$ $\epsilon$-close to $1$ in $C^1$, the vector fields $(\phi_1)_*X$ and $hX'$ get to agree on $\phi_1(K)$.
Our work shows that $\iota$ is a $\partial$-strong Birkhoff section for $(\phi_1)^*(hX')$. 
It follows that $\phi_1 \circ \iota$ is a $\partial$-strong Birkhoff section for $X'$ when $X'\in \mathcal{N}_\epsilon (X)$.
\end{proof}

\section{Genericity of entropy}

Let $\phi^t:M\to M$ be a  flow on a closed $3$-manifold.  
The topological entropy $h_{top}(\phi^t)$ is a non-negative number that measures the complexity of the flow. 
We review a definition, originally introduced by Bowen \cite{Bow}. 
We first endow $M$ with a metric $d$. 
Given $T>0$ we define $d_T(x,y)=\max \{ d(\phi^t(x),\phi^t(y)) \colon t\in [0,T] \}$ for any couple of points $x,y\in M$. 
A subset $S\subset M$ is $(T,\epsilon)$-separated if for all $x\neq y $ in $S$ we have that $d_T(x,y)>\epsilon$. 
Let $N(T,\epsilon)$ be the maximal cardinality of a $(T,\epsilon)$-separated subset of $M$. 
The topological entropy is defined as
$$
h_{top}(\phi^t) = \lim_{\epsilon\to 0} \ \limsup_{T\to \infty}\frac{1}{T}\log(N(T,\epsilon)).
$$
If $M$ has dimension 3 and the flow is at least $C^2$, deep results of Katok~\cite{katok} adapted to the case of flows by Lima and Sarig~\cite{LS19}, imply that having positive entropy is equivalent to the existence of a transverse homoclinic connection: an intersection between the stable and unstable manifolds of a hyperbolic periodic orbit that is transverse. 
Thus, having positive topological entropy is an open condition in the $C^\infty$-topology.
Consequently, to establish Theorem~\ref{thm: entropy} it remains to show that every contact form on a fixed closed contact $3$-manifold can be arbitrarily well $C^\infty$-approximated by a contact form that defines the same contact structure, and whose Reeb flow has positive topological entropy.

\subsection{Input from topological surface dynamics}

Here we review results on the dynamics of surface homeomorphisms due to Mather~\cite{Mather}, Koropecki~\cite{Koro2010}, Koropecki, Le Calvez and Nassiri~\cite{KLN}, and Le Calvez and Sambarino~\cite{LS}, as well as some basic notions in Carath\'eodory's theory of prime ends.

In this section $S$ is an orientable surface without boundary.
A boundary representative of $S$ is a sequence $P_1 \supset P_2 \supset P_3 \dots$ of connected, unbounded (not relatively compact), open sets such that $\partial P_i$ is compact for every $i$, and such that for every $K \subset S$ compact there exists $n_K$ such that $i>n_K$ implies $P_i \cap K = \emptyset$.
Here $\partial P_i$ denotes the topological boundary of $P_i$ relative to $S$.
Two boundary representatives $\{P_i\}$ and $\{P'_i\}$ are equivalent if for every $n$ there exists $m$ such that $P_m \subset P'_n$, and \textit{vice versa}.
An ideal boundary point is an equivalence class of boundary representatives.
The set of ideal boundary points of~$S$, also called the ideal boundary of~$S$, is denoted by $b_IS$.
The ideal completion of $S$ is defined by $c_IS = S \sqcup b_IS$, with the topology generated by the open subsets of $S$, and sets of the form $V \cup V'$ where $V$ is an open subset of $S$ whose boundary relatively to $S$ is compact, and $V' \subset b_IS$ consists of points that admit a boundary representative $\{P_i\}$ satisfying $P_i \subset V$ for every~$i$.
If $f:S\to S$ is a homeomorphism then there is an induced homeomorphism $f_S:c_IS\to c_IS$.
If $b_IS$ is finite then $c_IS$ is an orientable compact surface without boundary.

Let $U \subset S$ be open. 
The impression of $p \in b_IU$ in $S$ is the set 
$$
Z(p) = \bigcap_{\substack{V \subset c_IU \ \text{open} \\ p\in V}} {\rm cl}_S(V\cap U)
$$
where ${\rm cl}_S$ denotes closure relative to $S$. Note that $Z(p)$ is closed in $S$ and contained in the topological boundary $\partial U$ of $U$ relative to $S$.
If $S$ has finite genus then an ideal boundary point $p\in b_IU$ is said to be regular if $p$ is isolated in $b_IU$ and $Z(p)$ has more than one point.

A compact set $K \subset S$ is called a continuum if it is connected and has at least two points. 
A continuum $K$ is said to be cellular if $K = \cap_{n\in \N}D_n$ where each $D_n \subset S$ is a closed disk, and $D_{n+1}$ is contained in the interior of $D_n$, for every~$n$.
Cellular continua are contractible in $S$, see~\cite[page~411]{KLN}.
A continuum $K$ is said to be annular if $K = \cap_{n\in \N}A_n$, where each $A_n \subset S$ is homeomorphic to a closed annulus, $A_{n+1}$ is contained in the interior of $A_n$, and $K$ separates both boundary components of $A_n$, for every $n$.

From now on $S$ is assumed to have finite genus.
Let $U \subset S$ be open and $p\in b_IU$ be regular.
In Carath\'eodory's theory of prime ends one constructs a space $b_{\mathscr{E}_p}U$ homeomorphic to a circle, in such a way that if $U$ is invariant by a
homeomorphism $f:S\to S$ then there is an induced 
homeomorphism on $b_{\mathscr{E}_p}U$. Its Poincar\'e's rotation number is denoted by $\rho(f,p) \in \R/\Z$.
The reader is referred to~\cite[Section~3]{KLN} for a nice introduction to the theory of prime ends.

From now on let $f:S\to S$ be an orientation preserving homeomorphism.
A $1$-translation arc $\gamma$ for $f$ is defined to be an embedded compact arc from a point $x$, which is not a fixed point of $f$, to $f(x)$ such that $f(\gamma) \cap \gamma$ is either equal to $\{f(x)\}$ or equal to $\{x,f(x)\}$, the latter case precisely when $x$ is fixed by $f^2$.
If $N\geq 2$ then a $N$-translation arc for $f$ is a $1$-translation arc such that $f^k(\gamma)\cap\gamma = \emptyset$ for all $2\leq k \leq N-1$, and $f^N(\gamma) \cap \gamma$ is either equal to $\emptyset$ or to $\{x\}$, the latter case precisely when $x$ is fixed by $f^{N+1}$.

\begin{theorem}[Special case of Theorem~4.2 from~\cite{KLN}]
\label{thm_4.2_KLN}
Suppose that $S$ is compact, and that $U \subset S$ is an open $f$-invariant disk such that $S \setminus U$ has more than one point.
Assume that $f$ preserves a Borel measure $\mu$ on $S$ such that $\mu(W)>0$ for every non-empty open set $W \subset S$, and that $\mu(U\setminus C)<\infty$ for some compact set $C \subset U$.
Let $p$ be the point in $b_IU$ and assume that $\rho(f,p) \neq 0$.
Then there exists $N \in \N$ and a compact set $K \subset U$ such that every $N$-translation arc in $S\setminus K$ is disjoint from $\partial U$.
\end{theorem}

The assumption that $f$ preserves a Borel measure that is positive on non-empty open sets and  finite on $U\setminus C$ for some compact $C\subset U$, implies that $f$ is $\partial$-nonwandering in the sense~\cite[Definition~3.10]{KLN}. The latter property is an assumption of~\cite[Theorem~4.2]{KLN}.

\begin{theorem}[Corollary~7.2 from~\cite{KLN}]
\label{thm_cor_7.2_KLN}
Suppose that $S$ is compact, $f$ is non-wandering, $U\subset S$ is an $f$-invariant open connected set, and $p\in b_IU$ is regular and fixed.
If $\rho(f,p)$ is irrational then either $Z(p)$ is an annular continuum with no periodic points, or $Z(p)$ is a cellular continuum with a unique fixed point and no other periodic points.
\end{theorem}

To conclude this section we review the notions of Moser stable and Mather sectorial periodic points from~\cite[Section~5]{Mather}.
Let $q$ be a fixed point of $f$.

Choose a contracting homeomorphism $\alpha$ of $[0,+\infty)$, i.e. $\alpha^n(t) \to 0$ for all $t\geq0$.
The map $(s,t) \mapsto (\alpha(s),\alpha^{-1}(t))$ defines a homeomorphism of $[0,+\infty) \times [0,+\infty)$ denoted by $\alpha \times \alpha^{-1}$.
An elementary sector for $(f,q)$ is a closed subset $U \subset S$ such that $q \in \partial U$, $q$ has a neighborhood $N$ in $U$ satisfying $f(N) \subset U$ and $f^{-1}(N) \subset U$, and the germ of $f|_U$ at $q$ is topologically conjugate to the germ of $\alpha \times \alpha^{-1}$ at $(0,0)$.
Then $q$ is said to be a Mather sectorial fixed point of $f$ if a neighborhood of $q$ in $S$ is a finite union of elementary sectors for $(f,q)$.
It turns out that this definition is independent of the choice of $\alpha$.

The point $q$ is a Moser stable fixed point of $f$ if every neighborhood of $q$ in $S$ contains an $f$-invariant disk $D$ with $q$ in its interior, such that $f|_{\partial D}$ has an orbit dense in $\partial D$.

If $q$ is a periodic point of $f$, then $q$ is a Mather sectorial periodic point of $f$ if it is a Mather sectorial fixed point of $f^n$ for some $n$. 
Similarly, $q$ is a Moser stable periodic point of $f$ if it is a Moser stable fixed point of $f^n$ for some $n$.

\begin{remark}
\label{rmk_translation_arcs}
Let $q$ be a Mather sectorial periodic point of $f$. 
It follows from the above definitions that there exists $n \geq 1$ such that $f^n(q)=q$, and such that for every $N\geq 1$ there is neighborhood $V_N$ of $q$ in $S$ with the following property: every point in $V_N \setminus \{q\}$ belongs to an $N$-translation arc for $f^n$.
\end{remark}

We continue to follow~\cite[Section~5]{Mather} closely.
A fixed connection of $f$ is an $f$-invariant arc $\gamma \subset S$ whose end points are fixed points of $f$, or an $f$-invariant circle $\gamma$ such that $f|_\gamma$ is orientation preserving and $\gamma$ contains a fixed point of $f$. 
A periodic connection is a fixed connection of $f^n$, for some $n$.

Let $\mu$ be a Borel probability measure on $S$. 
We denote by $\mathcal{A}(S,\mu)$ the set of orientation preserving homeomorphisms $f:S\to S$ satisfying $f_*\mu=\mu$.
We denote by $\mathcal{G}(S,\mu)$ the subset of $\mathcal{A}(S,\mu)$ consisting of those homeomorphisms such that every periodic point is Mather sectorial or Moser stable, and have no periodic connections.

\subsection{Applications to return maps of Birkhoff sections}

Our goal here is to establish the following result.

\begin{proposition}
\label{prop_no_elliptic_homoclinic}
Let $\lambda$ be a strongly nondegenerate contact form defined on a closed and connected $3$-manifold $M$.
If its Reeb flow has no elliptic periodic orbits, and has a $\partial$-strong Birkhoff section, then some periodic orbit has a homoclinic connection.
\end{proposition}

We are now concerned with the proof of Proposition~\ref{prop_no_elliptic_homoclinic}.
The Birkhoff section is an immersion $\iota:S\to M$ defined on a compact orientable surface with boundary $S$ as in Definition~\ref{defn_BS_GSS}.
As explained in Section~\ref{sec_getting_BSs}, we can blow the link $L = \iota(\partial S)$ up and obtain a $3$-manifold $D_L$ with boundary, whose boundary components consist of invariant tori.
It is straightforward to check that there is a unique immersion $\widehat\iota:S \to D_L$ that agrees with $\iota$ on $\dot S = S\setminus\partial S$, maps $\partial S$ to $\partial D_L$ and is transverse to $\partial D_L$ along $\partial S$.
As in the proof of Proposition~\ref{prop: open}, $\widehat\iota$ defines a smooth embedding $S\hookrightarrow D_L$ transverse to $\partial D_L$ that defines a global section for the extended flow on $D_L$, and is still denoted by $S$ for simplicity.
The return map on~$S$ is a smooth diffeomorphism.
The space $c_I\dot S$ is nothing but the closed orientable surface obtained from $S$ by collapsing the boundary components to points.
The return map on~$S$ induces an orientation preserving homeomorphism 
$$
f : c_I\dot S \to c_I\dot S
$$
that is smooth on $\dot S$, and such that $b_I\dot S$ consists of periodic points.
If $\lambda$ is the contact form defining the Reeb flow on $M$, then $d\lambda$ defines a finite Borel measure on $\dot S$ that is invariant by $f|_{\dot S}$. After normalizing, we get an induced Borel probability measure $\mu$ on $c_I\dot S$ with the following properties: 
\begin{itemize}
\item $\mu$ is $f$-invariant. 
\item Every non-empty open subset $W \subset c_I\dot S$ satisfies $\mu(W)>0$.
\item $\mu$ agrees on $\dot S$ with a constant multiple of the measure induced by $d\lambda$.
\end{itemize}
Since every component $\gamma$ of $L$ is a hyperbolic periodic orbit of the Reeb flow, the embedded surface $S \subset D_L$ could be modified in such a way that all points in $b_I\dot S$ are Mather sectorial periodic points of $f$, keeping all the other properties described so far.
This follows from a normal form of the Reeb flow near the hyperbolic closed Reeb orbit $\gamma$.
Clearly, every periodic point in $\dot S$ coming from a periodic orbit of the Reeb flow in $M\setminus L$ is Mather sectorial, since all closed Reeb orbits are hyperbolic. 
Moreover, the strong nondegeneracy assumption implies that there are no periodic connections. Summarizing, we have $f \in \mathcal{G}(c_I\dot S,\mu)$.
If follows that $f^n \in \mathcal{G}(c_I\dot S,\mu)$ for every $n\in\Z$.

The next statement is a direct application of~\cite[Theorem~5.1]{Mather} and of Theorem~\ref{thm_cor_7.2_KLN}.
Its proof is extracted from the proof of~\cite[Theorem~8.3]{KLN}.

\begin{lemma}
\label{lemma_thm_8.3_KLN}
Let $U \subset c_I\dot S$ be an open $f$-invariant connected set, and let $p \in b_IU$ be regular and periodic.
Then $Z(p) \subset c_I\dot S$ is an annular continuum with no periodic points of $f$. 
\end{lemma}

\begin{proof}
Up to taking powers, there is no loss of generality to assume that $U$ is invariant and $p$ is fixed. A key tool is~\cite[Theorem~5.1]{Mather} asserting that if $g$ is an area- and orientation-preserving homeomorphism of an orientable surface such that every periodic point is either Moser stable or Mather sectorial with no periodic connections, if $V$ is a $g$-invariant open set, and if $q \in b_IV$ is regular and fixed by the map on $c_IV$ induced by $g$, then $\rho(g,q)\neq 0$ in $\R/\Z$.
Applying this fact to $U$ and all iterates of $f$ we conclude that $\rho(f,p)$ is irrational.
Theorem~\ref{thm_cor_7.2_KLN} implies that either $Z(p)$ is an annular continuum with no periodic points, or it is a cellular continuum that has a fixed point $x$ of $f$ and no other periodic points of~$f$. 
Note that $Z(p)$ has empty interior since it is contained in the boundary of $U$.
Now we reproduce the argument from~\cite[Theorem~8.3]{KLN} with small adaptations, and assume by contradiction that the latter alternative holds.

The surface $c_I\dot S$ cannot be a sphere.
Otherwise $c_I\dot S \setminus Z(p)$ would be an open $f$-invariant disk with irrational prime ends rotation number. The fixed point $x$ in $Z(p)$ is Mather sectorial.  
We can use Remark~\ref{rmk_translation_arcs} to find, for every $N$, an $N$-translation arc $\gamma$ for $f$ inside any neighborhood of $x$, 
such that $\gamma$ intersects $Z(p)$.
This contradicts Theorem~\ref{thm_4.2_KLN}; note that $c_I\dot S \setminus U$ is not a point since $Z(p)$ is a continuum.

It follows that the universal covering $\pi : \R^2 \to c_I\dot S$ is a plane.
Choose a point $\tilde x \in \pi^{-1}(x)$, and a lift $\tilde f$ of $f$ such that $\tilde f(\tilde x)=\tilde x$.
Let $K$ be the connected component of $\pi^{-1}(Z(p))$ contained in a disk $\tilde D$ having $\tilde x$ in its interior, and that $\pi|_{\tilde D}$ is a homeomorphism onto a neighborhood of $Z(p)$.
Then $\tilde f(K)=K$ and $K$ has empty interior.
Let $S'$ be the one-point compactification of $\R^2$, i.e. $S'$ is the sphere obtained from $\R^2$ by adding a point at infinity, and let $f':S'\to S'$ be the map induced by $\tilde f$.
Then $U' = S' \setminus K$ is an open disk with $\partial U'=K$ (boundary relative to $S'$).
If $p'$ is the (unique) point of $b_IU'$ then $\rho(f',p') = \rho(f,p)$ is irrational.
But $\tilde x$ is Mather sectorial and we can argue as above, using Remark~\ref{rmk_translation_arcs}, to find, for every $N$, an $N$-translation arc $\gamma$ for 
$f'$ inside any neighborhood of $\tilde x$, 
such that $\gamma$ intersects $K$.
This contradicts Theorem~\ref{thm_4.2_KLN}.
\end{proof}

\begin{corollary}[Corollary~8.7 from~\cite{KLN}]
\label{cor_8.7_KLN}
If $U \subset c_I\dot S$ is an open connected $f$-periodic set with $\#b_IU<\infty$, then the boundary of $U$ is a finite disjoint union of aperiodic annular continua and periodic points.
\end{corollary}

\begin{proof}
The set $b_IU$ consists of finitely many periodic points for the induced map $f_U:c_IU\to c_IU$. Moreover, every point in $b_IU$ is an isolated point.
Indeed, let $p=[(P_i)_{i\in \N}]$ and $q=[(Q_j)_{j\in \N}]$ be two points in $b_IU$.
If there exists $i_0,j_0$ such that $P_{i_0} \cap Q_{j_0}$ is relatively compact in $S$, then there exists $i_1,j_1$ such that for every $i>i_1$ and $j>j_1$, $P_i\cap Q_j =\emptyset$ (since by definition $P_i$ avoids every fixed relatively compact set for $i$ large enough) and thus $p$ and $q$ are separated.
Otherwise $P_i\cap Q_j$ is unbounded for all $i,j$.
We decompose $P_i = (P_i \setminus {\rm cl}_S(Q_j)) \sqcup (P_i\cap \partial_S Q_j) \sqcup (P_i \cap Q_j)$. For $i$ large enough, $P_i$ avoids the compact $\partial_S Q_j$ and, as 
$P_i$ is connected, it is contained either in the open set $P_i \setminus {\rm cl}_S(Q_j)$ or in the open set $P_i\cap Q_j$. It is necessarily in $P_i\cap Q_j$ since  $P_i\cap Q_j$ is unbounded (hence nonempty) and thus $P_i \subset Q_j$ when $i$ is large enough.
Exchanging the role of $p$ and $q$ we also get that given $i$, if $j$ is large enough $Q_j\subset P_i$. Putting the two inclusions together, we get $p=q$.

We find $n\geq 1$ such that $U$ is invariant by $f^n$, and every point in $b_IU$ is fixed by $(f_U)^n$.
Let $p\in b_IU$. If $Z(p)$ is not a point then $p$ is regular and, by Lemma~\ref{lemma_thm_8.3_KLN}, $Z(p)$ is an aperiodic annular continuum invariant by $f^n$.
If $Z(p)$ is a point then it is a periodic point of $f$. The conclusion follows since the boundary of $U$ is $\cup_{p\in b_IU}Z(p)$.
\end{proof}

\begin{corollary}[Corollary~8.9 from~\cite{KLN}]
\label{cor_8.9_KLN}
Let $x \in \dot S$ be a periodic point of $f$, which is necessarily hyperbolic.
If $x$ belongs to some periodic continuum $K \subset c_I\dot S$, then the stable and unstable manifolds of $x$ are contained in $K$.
Moreover, all four stable and unstable branches of $x$ have the same closure in $c_I\dot S$.
\end{corollary}

\begin{proof}
Let $\gamma$ be a branch of $x$, stable or unstable, such that $\gamma \not\subset K$.
Then $\gamma$ intersects a connected component $V$ of $c_I\dot S \setminus K$ such that $x\in \partial V$.
The set $V$ is periodic since $f$ preserves $\mu$. By~\cite[Lemma~2.3]{Mather}, see also~\cite[page~457]{KLN}, $b_IV$ is finite.
By Corollary~\ref{cor_8.7_KLN}, $\partial V$ is a disjoint finite union of periodic points and aperiodic annular continua.
But the connected component of $\partial V$ containing $x$ is not an aperiodic continuum since it contains the periodic point $x$.
Hence $x$ is an isolated point of $\partial V$, in particular also of $K$, in contradiction to the fact that $K$ is a continuum.

Let $\gamma_1,\gamma_2$ be branches of $x$.
The closure $K$ of $\gamma_1$ in $c_I\dot S$ is a periodic continuum.
Hence, by what was proved above, $\gamma_2 \subset K$. 
It follows that the closure of $\gamma_2$ is contained in the closure of $\gamma_1$.
The same argument interchanging the roles of $\gamma_1$ and $\gamma_2$ concludes the proof.
\end{proof}

With the above results in place, we can consider a relation on the set of periodic points of $f$ contained in $\dot S$.
Note that, by assumption, all such periodic points are hyperbolic.
Similarly to~\cite[Section~3]{LS}, a periodic point $x_1$ is related to a periodic point $x_2$ if each of the four stable/unstable branches of $x_1$ has the same closure as any of the four branches of $x_2$.
This relation is clearly symmetric and transitive.
Reflexivity is non-trivial and follows from Corollary~\ref{cor_8.9_KLN}.
Hence, this is an equivalence relation.
The set of equivalence classes is denoted by $\mathcal{E}(f)$.
For a given $\kappa \in \mathcal{E}(f)$ we denote by $K(\kappa)$ the closure of one (hence any) of the four branches of a point in $\kappa$.

\begin{lemma}[Corollary~3.2 from~\cite{LS}]
\label{lemma_3.2_LS}
If $\kappa \in \mathcal{E}(f)$ and $x \in \dot S$ is a (hyperbolic) periodic point such that $x\in K(\kappa)$, then $x\in \kappa$.
\end{lemma}

\begin{proof} 
Let $\kappa' \in \mathcal{E}(f)$ denote the class of $x$.
By Corollary~\ref{cor_8.9_KLN}, $K(\kappa') \subset K(\kappa)$.
Suppose by contradiction that  $K(\kappa) \not\subset K(\kappa')$.
Then there is a connected component~$V$ of $c_I\dot S \setminus K(\kappa')$ such that $K(\kappa) \cap V \neq \emptyset$.
In particular, $V$ intersects some branch of~$x$.
Choose $n$ such that $f^n(V)=V$ and $f^n(x)=x$.
It follows that $x \in \partial V$.
By~\cite[Lemma~2.3]{Mather}, see also~\cite[page~457]{KLN}, $b_IV$ is finite, in particular $b_IV$ consists of isolated points, and for every $p\in b_IV$ the set $Z(p)$ is not a point. 
It follows from Lemma~\ref{lemma_thm_8.3_KLN} that for every $p\in b_IV$ the set $Z(p)$ is an annular continuum with no periodic point. 
But for some $p$ the set $Z(p)$ contains $x$, and this contradiction proves that $K(\kappa) \subset K(\kappa')$.
We are done showing that $K(\kappa) = K(\kappa')$, from where it follows that $\kappa=\kappa'$.
\end{proof}

The above lemma was the last tool we needed to be able to conclude that the proof of~\cite[Proposition~5.1]{LS} can be reproduced~\textit{ipsis litteris} to establish the following statement.
We do not reproduce the argument here.

\begin{lemma}
\label{lemma_5.1_LS}
Let $g$ be the genus of $c_I\dot S$.
Every set of (necessarily hyperbolic) periodic points in $\dot S$ with strictly more than $2g$ elements contains at least one point that has a homoclinic connection.
\end{lemma}

We can now finish the proof of Proposition~\ref{prop_no_elliptic_homoclinic}. 
Since the contact form $\lambda$ is nondegenerate, a result from~\cite{CDR} implies that there are either two or infinitely many periodic Reeb orbits.
If there are exactly two periodic orbits then the results from~\cite{CGHHL} imply that both are elliptic.
Since we work under the assumption that there are no elliptic periodic Reeb orbits, it follows that there are infinitely many such periodic orbits.
These give rise to infinitely many periodic points of $f$ in $\dot S$.
From Lemma~\ref{lemma_5.1_LS} we get the desired homoclinic connection.

\subsection{Proof of Theorem~\ref{thm: entropy}}

Fix a co-orientable contact structure on $M$.
Consider the set $\Lambda$ of contact forms defining $\xi$, equipped with the $C^\infty$-topology. 
This set is an open subset of the vector space $\Omega^1(M)$ equipped with the $C^\infty$-topology. 
The latter is a topological vector space whose topology can be defined by a complete metric.
Consider the set $\Lambda_+ \subset \Lambda$ of contact forms whose Reeb flows have positive topological entropy.
As mentioned before, $\Lambda_+$ is open, and we need to show that $\Lambda_+$ is dense in $\Lambda$.

Let $\Lambda_* \subset \Lambda$ denote the set of nondegenerate contact forms. 
As is well-known,~$\Lambda_*$ contains a countable intersection of open and dense subsets of $\Lambda$.
In particular, $\Lambda_*$ is dense in $\Lambda$.
Let  $\Lambda^e \subset \Lambda$ denote the subset of all  contact forms that admit at least one elliptic closed Reeb orbit.
Then $\Lambda^h = \Lambda_* \setminus \Lambda^e$ consists precisely of those contact forms on $\Lambda$ all of whose periodic Reeb orbits are hyperbolic.

Let $\lambda \in \Lambda_* \cap \Lambda^e$ and $\gamma$  an elliptic  periodic Reeb orbit of $\lambda$. The Poincar\'e map on a small section transverse to $\gamma$ can be described as a symplectic embedding $\varphi_0$ defined on an open neighborhood of $(0,0) \in \R^2$ into $\R^2$, such that $\varphi_0(0,0)=(0,0)$. 
Up to a linear symplectic change of coordinates, there is no loss of generality to assume that $D\varphi_0(0,0) = R_{\alpha_0}$ is rotation by an angle $\alpha_0 \in \R \setminus 2\pi\Q$. Therefore, we can write $\varphi_0 = R_{\alpha_0} \circ \psi_0$ for some symplectic map $\psi_0$ defined near the origin, satisfying $\psi(0,0)=(0,0)$, $D\psi_0(0,0)=I$. It follows that $\psi_0$ can be represented near $(0,0)$ by a generating function $u_0$. This means that
\begin{equation}
\label{gen_funct}
\psi_0(x,y)=(X,Y) \qquad \Leftrightarrow \qquad \left\{ \begin{aligned} X-x &= D_2u_0(x,Y) \\ y-Y &= D_1u_0(x,Y) \end{aligned} \right.
\end{equation}
Here $D_1$ and $D_2$ denote partial derivatives of $u_0$ with respect to first and second variables, respectively. The Hessian of $u_0$ at $(0,0)$ vanishes. 

Let $r \in \N$ be arbitrary. It follows~\cite[Corollary~1]{Zeh} that $u_0$ can be $C^r$-slightly perturbed on a small fixed neighborhood of~$(0,0)$ to a $C^r$-function~$u_1$, and that $\alpha_0$ can be slightly perturbed to $\alpha_1 \in \R \setminus 2\pi\Q$, in such a way that the following holds: If $\psi_1$ is the map induced by $u_1$ as in~\eqref{gen_funct}, then the map $\varphi_1 = R_{\alpha_1}\circ\psi_1$ has transverse homoclinic connections in every neighborhood of $(0,0)$.

Using bump functions one constructs a $C^r$ function $u$ defined on a small open neighborhood of $(0,0)$, that agrees with $u_1$ near $(0,0)$ and with $u_0$ outside of a small neighborhood of $(0,0)$, whose $C^r$-distance to $u_0$ is controlled by the $C^r$-distance between $u_1$ and $u_0$. Such a function can be used to construct a symplectic $C^r$-embedding $\varphi$ defined on a small open ball centered at $(0,0)$ that agrees with $\varphi_1$ near the origin and with $\varphi_0$ outside of a neighborhood of the origin. The map $\varphi$ can be embedded as Poincar\'e map of an elliptic periodic orbit of a contact form $\lambda'$ of class $C^{r+1}$.
By construction, the contact form $\lambda'$ has a hyperbolic periodic Reeb orbit with a homoclinic connection.
Hence, $\lambda'$ can be $C^{r+1}$-slightly perturbed to a smooth contact form $\lambda''$ that has a hyperbolic periodic Reeb orbit with a homoclinic connection.
In particular $\lambda'' \in \Lambda_+$.
Since $r\geq1$ is arbitrary and the size of the perturbations can be taken arbitrarily small, we find via the above process a sequence of smooth contact forms $\lambda_n \in \Lambda_+$ such that $\lambda_n \to \lambda$ in $C^\infty$.



Now let $\lambda \in \Lambda\setminus \Lambda^e$.
There is a sequence of contact forms $\lambda_n$ satisfying $\lambda_n \to \lambda$ in $C^\infty$, such that every $\lambda_n$ is strongly nondegenerate and has a $\partial$-strong Birkhoff section for its Reeb flow.
This is so because both properties are $C^\infty$-generic; the latter property holds in an open and dense set of contact forms in the $C^\infty$-topology by Corollary~\ref{cor_generic_existence}. 
If there exists a subsequence $n_j \to \infty$ such that $\lambda_{n_j} \in \Lambda_e$ then, by what was observed so far, $\lambda$ is a limit of contact forms in $\Lambda_+$.
Hence, it remains to deal with the case where there exists $N \in \N$ such that for every $n\geq N$ the contact form $\lambda_n$ is strongly nondegenerate, its Reeb flow has a $\partial$-strong Birkhoff section, and all periodic Reeb orbits are hyperbolic.
Proposition~\ref{prop_no_elliptic_homoclinic} implies that for every $n\geq N$ the Reeb flow of $\lambda_n$ has a hyperbolic periodic orbit with a homoclinic connection, which in turn implies that $\lambda_n \in \Lambda_+$.
The proof of Theorem~\ref{thm: entropy} is complete.

\appendix

\section{A result in Schwartzman-Fried-Sullivan theory}
\label{app_SFS}

In this appendix $M$ is a closed, connected and oriented $3$-manifold, and $X$ is a smooth vector field on $M$.
The flow of $X$ is denoted by $\phi^t$.
Let $L \subset M$ be a link formed by periodic orbits. 
As before, we denote by $\P_\phi(M \setminus L)$ the set of $\phi$-invariant Borel probability measures on $M \setminus L$.
The statement below can be found as Theorem~2.3 in the extended arXiv version of~\cite{icm}, but its proof was not available in the literature as far as we know, except for some arguments sketched by Ghys~\cite{ghys}, and for a version stated in terms of homology directions by Fried~\cite{fried}.

\begin{theorem}
\label{thm_existence_from_SFS}
Suppose that there exists $y \in H^1(M \setminus L;\R)$ satisfying 
\begin{itemize}
\item $\mu \cdot y > 0$ for all $\mu \in \P_\phi(M \setminus L)$.
\item $\rho^y(\gamma) > 0$ for all $\gamma \subset L$.
\end{itemize}
Then there exists a $\partial$-strong Birkhoff section for $\phi^t$ whose boundary components are contained in $L$.
\end{theorem}

\begin{remark}
A refinement of Theorem~\ref{thm_existence_from_SFS} with conditions for global surfaces of section representing a prescribed homology class can be found in~\cite{Hry_SFS}.
\end{remark}

\begin{remark}
Some components of $L$ may not be contained in the boundary of the  Birkhoff section provided by Theorem~\ref{thm_existence_from_SFS}.
\end{remark}

\subsection{Preliminaries}

\subsubsection{Blowing periodic orbits up}

As proved in Section~\ref{ssec_blowing_up}, we can blow $L$ up and construct a new manifold without boundary~$M_L$~\eqref{blown_up_manifold}, containing a special smooth compact domain $D_L$~\eqref{domain_D_L}.
The domain $D_L \subset M_L$ is the closure of $M\setminus L$ in $M_L$.
Moreover, $X$ can be extended from $M \setminus L$ to $M_L$ as a smooth vector field $X_L$. 
The vector field $X_L$ is not unique, but its restriction to $D_L$ is unique. 
Moreover, $X_L$ is tangent to $\partial D_L$ and the dynamics on $\partial D_L$ captures the linearized dynamics along the components of $L$ in a precise way described in Section~\ref{ssec_blowing_up}.

For later purposes we need to fix some notation.
Denote the components of $L$ by $\gamma_1,\dots,\gamma_h$ and their primitive periods by $T_j>0$. 
For each $j=1,\dots,h$, $\Sigma_j \subset \partial D_L$ denotes the boundary torus associated to the end of $M\setminus L$ near $\gamma_j$.

\subsubsection{Schwartzman cycles and structure currents}

For every $p\geq 0$ one can turn $\Omega^p(M_L)$ into a topological vector space by equipping it with the $C^\infty_{\rm loc}$-topology.
Its topological dual is denoted by $C_p = \Omega^p(M_L)'$ and equipped with the weak* topology. 
As mentioned before, an element of $C_p$ is a $p$-current with compact support.
Denote by $C_p'$ the topological dual of $C_p$ equipped with its weak* topology. 
The map $\Omega^p(M_L) \to C_p'$ given by $\omega \mapsto \left<\cdot,\omega\right>$ is a linear homeomorphism, in other words $\Omega^p(M_L)$ is reflexive. 
There is a boundary operator $\partial : C_{p+1} \to C_p$ defined as the adjoint of the exterior derivative $d:\Omega^p(M_L) \to \Omega^{p+1}(M_L)$. 
A current in $C_p$ is called a cycle if it is in the kernel of $\partial :C_p \to C_{p-1}$, and is called a boundary if it is in the image of $\partial: C_{p+1} \to C_p$. 
The space of boundaries is denoted by $B_p$, the space of cycles by $Z_p$, and we have $H_p(M_L;\R) = Z_p/B_p$.

Consider the set $\P(M_L)$ of compactly supported finite Borel measures on~$M_L$, and let $\P_{\phi_L}(D_L) \subset \P(M_L)$ be the subset of those which are $\phi_L$-invariant probability measures supported on~$D_L$. 
Any $\mu \in \P(M_L)$ defines a $1$-current $c_\mu \in C_1$ by the formula 
\[
\left< c_\mu,\omega \right> = \int_{M_L} \omega(X_L) \ d\mu, \qquad \omega\in\Omega^1(M_L) \, .
\]
We follow Sullivan's notation and write $c_\mu = \int_{M_L} X_L \ d\mu$. 
If $\mu$ is supported on~$D_L$ then $c_\mu$ is a cycle if, and only if, $\mu$ is $\phi_L$-invariant. 
The elements of the set 
\begin{equation*}
\s_X := \{ c_\mu \mid \mu\in \P_{\phi_L}(D_L) \} \subset Z_1
\end{equation*}
will be called {\it Schwartzman cycles}.
The {\it Dirac currents} $\delta_p\in C_1$, $p\in M_L$, are defined by $\left<\delta_p,\omega\right> = \omega(X_L)|_p \in \R$. 
Let $\Cc\subset C_1$ denote the closed convex cone generated by $\{\delta_p \mid p\in D_L\}$. 
In~\cite{sullivan} $\Cc$ is called the {\it cone of structure currents in $D_L$}.
The following lemma summarizes the analytical facts we need.

\begin{lemma}[\cite{Hry_SFS}, Lemma~3.3 and Lemma~3.4]
\label{lemma_analysis}
The following hold.
\begin{itemize}
\item[(I)] There exists $\omega \in \Omega^1(M_L)$ such that $\left< c,\omega \right>>0$ for every $c\in \Cc\setminus\{0\}$.
\item[(II)] If $\omega$ is as in (I) then the convex set $K = \{ c\in\Cc \mid  \left< c,\omega \right>=1\}$ is compact.
\item[(III)] For every $c\in\Cc$ there exists a unique finite Borel measure~$\mu$ on $D_L$ such that $c=\int_{D_L}X_L \ d\mu$.
\end{itemize}
\end{lemma}

\subsection{Transverse foliations}

Let $\beta$ be a closed $1$-form $\beta\in\Omega^1(M_L)$ that represents~$y$ on $M\setminus L = D_L \setminus \partial D_L$.
We claim that 
\begin{equation}\label{basic_assumption}
\left< c_\mu,\beta \right> > 0 \qquad \forall \mu\in \P_{\phi_L}(D_L) \, .
\end{equation}
To see this, let $\mu\in\P_{\phi_L}(D_L)$ be arbitrary. For every Borel set $E\subset M_L$ define
\begin{equation*}
\mu_j(E) = \mu(E\cap \Sigma_j) \qquad \dot\mu(E) = \mu(E\cap (D_L\setminus \partial D_L)) = \mu(E\cap (M\setminus L)) \, .
\end{equation*}
Then $\dot\mu$ and the $\mu_j$ are $\phi_L$-invariant Borel measures, and $\mu=\dot\mu+\sum_j\mu_j$. We have
\begin{equation}\label{sum_decomp_measures}
\left< c_\mu,\beta \right> = \int_{M_L} \beta(X_L) \ d\mu = \int_{M_L} \beta(X_L) \ d\dot\mu + \sum_{j=1}^h \int_{M_L} \beta(X_L) \ d\mu_j \, .
\end{equation}
If $\mu_j(M_L) = \mu(\Sigma_j)>0$ then $\mu_j/\mu_j(M_L) \in \P_{\phi_L}(D_L)$ is supported on $\Sigma_j$. 
By Lemma~\ref{lemma_rot_numbers_measures} and the hypotheses of Theorem~\ref{thm_existence_from_SFS} 
\[
\int_{M_L} \beta(X_L) \ d\mu_j = \int_{\Sigma_j} \beta(X_L) \ d\mu_j = \frac{2\pi}{T_j} \mu_j(M_L) \rho^y(\gamma_j) > 0 \, .
\]
If $\dot\mu(M_L) = \mu(M\setminus L) >0$ then $\dot\mu/\dot\mu(M_L)$ induces an element of $\P_\phi(M\setminus L)$. 
By the hypotheses of Theorem~\ref{thm_existence_from_SFS}
\[
\int_{M_L} \beta(X_L) \ d\dot\mu = \int_{M\setminus L} \beta(X) \ d\dot\mu > 0 \, . 
\]
Thus each term in the sum~\eqref{sum_decomp_measures} is non-negative, and at least one term is positive since $\mu(M_L)=\mu(D_L)=1$. We have established~\eqref{basic_assumption}.

By (I) and (II) in Lemma~\ref{lemma_analysis} there exists $\omega \in \Omega^1(M_L)$ such that $\left<\cdot,\omega\right> > 0$ on $\Cc\setminus\{0\}$, and $K = \{c\in\Cc \mid \left<c,\omega\right>=1 \}$ is compact and convex in $C_1$. 
Let $c\in\Cc\setminus\{0\}$ be a cycle. 
By (III) in Lemma~\ref{lemma_analysis}, $c=\int_{M_L}X_L \ d\nu$ for some positive finite Borel measure~$\nu$ supported on~$D_L$, and $\nu$ must be invariant because $c$ is a cycle. 
In other words, $\mu := \nu/\nu(M_L) \in \P_{\phi_L}(D_L)$ and $c = \nu(M_L) c_\mu$ for a Schwartzman cycle $c_\mu$. 
From~\eqref{basic_assumption} we conclude that $\left<c,\beta\right> = \nu(M_L) \left< c_\mu,\beta \right> > 0$. In particular $\Cc \cap B_1 = \{0\}$, or equivalently $K \cap B_1 = \emptyset$, and $\beta$ evaluates positively on $K\cap Z_1$. By~\cite[Theorem~A.1]{Hry_SFS} we find $\eta_0 \in C_1'$ that vanishes on $B_1$, is positive on $K$, and agrees with $\beta$ on $Z_1$. By reflexivity $C_1'=\Omega^1(M_L)$ we conclude that $\eta_0$ is a $1$-form, and as such it must be closed since it vanishes on $B_1$. Moreover, ${\eta_0}|_{M\setminus L}$ represents~$y$ since it agrees with $\beta$ on $Z_1$. Finally, note that
\begin{equation}
\label{crucial_pointwise_ineq_SFS}
\eta_0(X_L)|_p = \left<\delta_p,\eta\right> > 0 \qquad \forall \ p \in D_L
\end{equation}
because $\left<\cdot,\eta_0\right>>0$ on $\Cc\setminus\{0\}$.

The kernel of $\eta_0$ integrates to a foliation transverse to $X_L$, and hence to $X$ on $M\setminus L$, but in general this foliation might behave very badly. 
In the next Section we shall deal with this issue by well-known arguments.

\subsection{Birkhoff sections}

Note that $H^1(M_L;\R) \simeq H^1(M\setminus L;\R)$ is a finite dimensional vector space. 
Hence, in view of~\eqref{crucial_pointwise_ineq_SFS}, we can find arbitrarily $C^\infty_{\rm loc}$-close to~$\eta_0$ a $1$-form $\eta_1$ on $M_L$ which is closed, has rational periods, and still satisfies~\eqref{crucial_pointwise_ineq_SFS}. 
Consider inclusions $\iota:D_L\hookrightarrow M_L$ and $\iota_j:\Sigma_j \hookrightarrow M_L$. 
The class $[\iota^*\eta_1] \in H^1(D_L;\R)$ is close to~$y$, and induces a class in $H^1(D_L;\Q)$.
Choose $m\in \N$ such that $\eta := m\eta_1$ has integer periods.
Note that $\eta$ satisfies~\eqref{crucial_pointwise_ineq_SFS}, that is
\begin{equation}
\label{crucial_pointwise_ineq_SFS_final}
\eta(X_L)|_p = \left<\delta_p,\eta\right> > 0 \qquad \forall \ p \in D_L \, .
\end{equation}
Denote by $N \in \N$ a generator of the group of periods of $\eta$. 
Then, after arbitrarily choosing a point $p_0 \in M \setminus L$, we can define a map ${\rm pr}:D_L \to \R/\Z$ by setting ${\rm pr}(p)$ to be the integral of $N^{-1}\iota^*\eta$ along any path from $p_0$ to $p$, modulo $\Z$:
\[
{\rm pr}(p) = \frac{1}{N} \int_{p_0}^p \iota^*\eta \mod \Z \, .
\]
The map ${\rm pr}:D_L \to \R/\Z$ is a smooth surjective submersion in view of~\eqref{crucial_pointwise_ineq_SFS_final}.
It follows from this construction that if $c:S^1 \to M\setminus L$ is a smooth loop then
\begin{equation}
\label{periods_formula}
\frac{1}{N} \int_c\eta = \text{degree of ${\rm pr}\circ c$} \, .
\end{equation}

The preimages ${\rm pr}^{-1}(x)$ are the leaves of a foliation of $D_L$ obtained by integrating $\ker \iota^*\eta$. 
Each ${\rm pr}^{-1}(x)$ is a compact embedded submanifold of $D_L$ that intersects the boundary $\partial D_L$ cleanly, since it is transverse to $X_L$ and $X_L$ is tangent to $\partial D_L$. 
It follows that ${\rm pr}^{-1}(x) \subset D_L$ is a smooth embedded surface transverse to $X_L$ with boundary equal to $\pr^{-1}(x)\cap \partial D_L$. 
Each leaf ${\rm pr}^{-1}(x)$ can be co-oriented by the vector field $X_L$, and can also be co-oriented by pulling back the canonical orientation of $\R/\Z$ via the map ${\rm pr}$. 
These co-orientations coincide by construction. 
Note that all trajectories in $D_L$ will hit all leaves ${\rm pr}^{-1}(x)$ in finite time, both in the future and in the past: this follows from compactness of $D_L$ and from~\eqref{crucial_pointwise_ineq_SFS_final}. 
In particular, for all $x\in\R/\Z$ the surface ${\rm pr}^{-1}(x)$ is a global surface of section for the flow of $X_L$ on $D_L$.

In a final step we modify one given fiber of the map ${\rm pr}$ to obtain a Birkhoff section for the flow of $X$ on $M$.
Let us fix $x \in \R/\Z$ arbitrarily.
For each $j \in \{1,\dots,h\}$, we choose coordinates $(t,\theta) \in \R/T_j\Z \times \R/2\pi\Z \simeq \Sigma_j$ as explained in Section~\ref{ssec_blowing_up}.
These can be used to write a basis $\{dt,d\theta\}$ of $H^1(\Sigma_j;\R)$.
Note that $N^{-1}\iota_{j}^*\eta$ is homologous to $a_1 \ dt/T_j + a_2 \ d\theta/2\pi$ for some $a_1,a_2\in\Z$. 
Assume that ${\rm pr}^{-1}(x) \cap \Sigma_j \neq \emptyset$, and let us consider a connected component $\alpha$ of ${\rm pr}^{-1}(x) \cap \Sigma_j$, oriented as part of the boundary of ${\rm pr}^{-1}(x)$.
Then $\alpha$ is non-trivial in $H_1(\Sigma_j;\Z)$, since otherwise it would bound a disk $D \subset \Sigma_j$ with $X_L \pitchfork \partial D$ thus forcing a singularity of $X_L$ on $\Sigma_j$. 
Let $\{e_1,e_2\}$ be the basis in $H_1(\Sigma_j;\Z)$ dual to $\{dt/T_j,d\theta/2\pi\}$ and write $\alpha = n_1e_1 + n_2e_2$ in homology. 
We already know that $(n_1,n_2) \neq (0,0)$.
Since $\alpha$ is an embedded loop in $\Sigma_j$ we also know that $(n_1,n_2)$ is primitive: if $d \in \N$ and $(n_1,n_2)/d \in \Z\times \Z$ then $d=1$.
Moreover, $0 = N^{-1} \int_\alpha \iota_j^*\eta = a_1n_1 + a_2n_2$.
The numbers $n_1,n_2$ depend only on $j$ and not on the choice of $\alpha$.
Using the transversality of the flow, we can deform ${\rm pr}^{-1}(x)$ near ${\rm pr}^{-1}(x) \cap \Sigma_j$ to obtain a surface $S$ that intersects $\partial D_L$ transversely, is transverse to $X_L$ up to the boundary, still is a global section for the flow on~$D_L$, and such that the following holds: if $n_1=0$ then $S$ coincides with finitely many annuli of the form $\{t=t_*, \ r\in[0,\epsilon)\}$ near $\Sigma_j$; if $n_1\neq0$ then $dt$ does not vanish tangentially to $S \cap \Sigma_j$. 
Now we collapse each $\Sigma_j$ to $\gamma_j$ to obtain again the manifold~$M$ from~$D_L$. 
This collapsing map can be defined by $(t,r,\theta) \mapsto (t,re^{i\theta})$ near each $\Sigma_j$.
In this process $S$ gets map to a Birkhoff section for the flow of~$X$ on~$M$. 
The transversality to the flow on $\partial D_L$ before applying the projection implies that we get a $\partial$-strong Birkhoff section.
In the notation above, if $n_1=0$ then each loop $\alpha$ gets mapped to an interior point of the obtained section where it intersects~$\gamma_j$ transversely, and if $n_1\neq 0$ then the orbit $\gamma_j$ gets covered $|n_1|$ times.

\begin{remark}
Since the Birkhoff section obtained is $\partial$-strong, the associated return time function is bounded.
\end{remark}

\section{A lemma on weak* convergence}

Let $X$ be a compact manifold and let $\mu$ be a regular\footnote{A Borel probability measure is regular if for every Borel set $E$ and every $\epsilon>0$ we find a compact set $K$ and an open set $U$ such that $K \subset E \subset U$, $\mu(K) > \mu(E) - \epsilon$ and $\mu(U) < \mu(E) + \epsilon$.} Borel probability measure on $X$. 
Let $\mu_n$ be a sequence of Borel probability measures satisfying $\mu_n \to \mu$ in the weak* topology. 

\begin{lemma}
\label{lemma_measures_conv}
If $V \subset X$ is open with $\mu(\partial V)=0$, and if $f:V\to\R$ is a continuous bounded function, then
$$
\int_V f \, d\mu_n \to \int_V f \, d\mu \, .
$$
\end{lemma}

\begin{proof}
We start with a claim.

\medskip

\noindent \textit{Claim.} For every $\eta>0$ there exists an open neighborhood $W$ of $\partial V$, and $n_0\geq 1$, such that $n\geq n_0 \Rightarrow \mu_n(W) < \eta$.

\medskip

To prove the claim, we use the regularity of $\mu$ to find an open neighborhood $W'$ of $\partial V$ such that $\mu(W') < \eta$. 
Now we find $\psi:X\to[0,1]$ continuous such that $\supp(\psi) \subset W'$ and $\psi|_W \equiv 1$ on some open neighborhood $W\subset W'$ of $\partial V$.
Then
$$
\eta > \mu(W') \geq \int_X \psi \, d\mu = \lim_{n\to\infty} \int_X \psi \, d\mu_n
$$
which implies that we can find $n_0$ such that 
$$
n\geq n_0 \Rightarrow \mu_n(W) \leq \int_X \psi \, d\mu_n < \eta
$$
as desired. The proof of the claim is complete.

\medskip

Let $\epsilon>0$ be fixed arbitrary. For any $\delta>0$ we can use the regularity of $\mu$ to find a compact set $K_0 \subset V$ such that $\mu(K_0) > \mu(V) - \delta/2$, and an open set $U_0 \supset \overline{V}$ such that $\mu(U_0) < \mu(\overline V) + \delta/2$.
By the claim proved above we find an open neighborhood $W$ of $\partial V$ and $n_0$ such that $\mu_n(W) < \delta$ for all $n\geq n_0$.
Consider sets
$$
K = K_0 \cup (V \setminus W), \qquad \qquad U = U_0 \cap (V \cup W) \, .
$$
Then $K \supset K_0$ is compact, $U \subset U_0$ is open, and $U \setminus K \subset W$.
Moreover, we can estimate
$$
\mu(U\setminus K) \leq \mu(U_0\setminus K_0) = \mu(U_0) - \mu(K_0) < \mu(\overline V) + \delta/2 - (\mu(V) - \delta/2) = \delta
$$
where the assumption $\mu(\partial V)=0$ was used to get the equality $\mu(V) = \mu(\overline V)$. 
In addition, we have
$$ 
n\geq n_0 \Rightarrow \mu_n(U\setminus K) \leq \mu_n(W) < \delta \, .
$$
Now consider $\varphi:X\to[0,1]$ continuous such that $\supp(\varphi) \subset U \setminus K$ and $\varphi\equiv1$ on some neighborhood of $\partial V$.
Let $h:X\to\R$ be the continuous function that agrees with $(1-\varphi)f$ on $V$ and vanishes on $X\setminus V$.
By the assumption that $\mu_n\to\mu$ weak*, we can find $n_1$ such that
$$
n\geq n_1 \Rightarrow \left| \int_X h \, d\mu_n - \int_X h \, d\mu \right| < \delta \, .
$$
We can estimate
$$
\begin{aligned}
\left| \int_V f \, d\mu_n - \int_V f \, d\mu \right| &= \left| \int_V \varphi f \, d\mu_n - \int_V \varphi f \, d\mu + \int_X h \, d\mu_n - \int_X h \, d\mu \right| \\
&\leq \left| \int_V \varphi f \, d\mu_n \right| + \left| \int_V \varphi f \, d\mu \right| + \left| \int_X h \, d\mu_n - \int_X h \, d\mu \right| \\
&\leq \left(  \sup |f|  \right) (\mu_n(U\setminus K)+\mu(U\setminus K)) + \left| \int_X h \, d\mu_n - \int_X h \, d\mu \right|
\end{aligned} 
$$
from where we see that if $n \geq \max\{n_0,n_1\}$ then 
$$
\left| \int_V f \, d\mu_n - \int_V f \, d\mu \right| \leq 2\delta \left(  \sup |f| \right) + \delta \, .
$$
Finally, note that we could have chosen $\delta>0$ satisfying $2\delta \left( \sup |f| \right) + \delta<\epsilon$. 
The proof of the lemma is complete.
\end{proof}

\section{An alternative proof of the $C^1$-density of the existence of a Birkhoff section}
\label{section: C}

In this appendix we sketch an alternative proof of the following direct consequence of Corollary~\ref{cor_generic_existence}.
As the proof takes a different path, we think it is worth mentioning~it. 

\begin{theorem}\label{thm: section}
If $(M,\xi)$ is a closed oriented $3$-manifold then the set of Reeb vector fields for $\xi$ on $M$ that admit a Birkhoff section is $C^1$-dense in the set of Reeb vector fields.
\end{theorem}

As we have mentioned in the introduction, a $C^\infty$-generic contact form $\lambda$ on~$(M,\xi)$ is nondegenerate. 
Let $R$ be the Reeb vector field of $\lambda$.
In that case, by Theorem~1.1 of \cite{CDR}, $R$ is carried by a broken book decomposition $(K,\mathcal{F})$. 
The binding $K$ of the broken book decomposition has radial and broken components, and if all the components are radial the broken book decomposition is a rational open book decomposition and any page is a Birkhoff section.
The proof of Theorem~4.13 in~\cite{CDR} provides a mechanism to successively eliminate broken components of the binding $K$, inspired by Fried~\cite{friedanosov}. 
It works as follows, see also Figure~\ref{figure:FriedPants}.

Let $k\in K$ be a broken component of the binding, which is automatically a hyperbolic orbit of the Reeb flow~$R$. 
Assume it is positive (i.e. the eigenvalues of the linearized first return map are positive) and denote by~$N$ and $S$ the two components of the complement of $k$ in its stable manifold, and by $E$ and $W$ the two components of the complement of $k$ in its unstable manifold (in red and blue respectively on Figure~\ref{figure:FriedPants} left).
If there is a transverse homoclinic connection $p$ from~$E$ to $N$ and a transverse homoclinic connection $q$ from $W$ to $S$, then one can find a periodic orbit~$p'$ of $R$ close to $p$ and another one~$q'$ close to $q$.
There is also a periodic orbit~$r'$ following successively $p$ and~$q$ thanks to the combinatorial description of the dynamics. 
In this situation, Fried~\cite{friedanosov} constructed a pair of pants~$P$ transverse to~$R$, bounded by these three orbits, and intersecting
the orbit~$k$ in its interior, see Figure~\ref{figure:FriedPants}. 

\begin{figure}[ht]
\centering
\begin{picture}(360,200)(0,0)
\put(-30,10){\includegraphics[width=.7\textwidth]{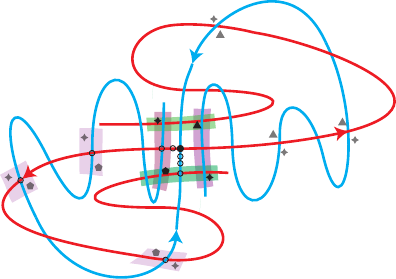}}
\put(220,0){\includegraphics[width=.4\textwidth]{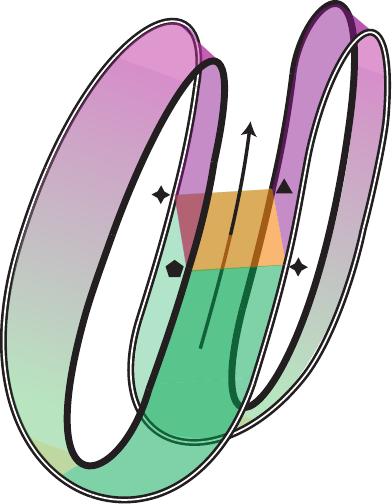}}
\put(310,142){$k$}
\put(330,181){$p'$}
\put(280,157){$q'$}
\put(260,172){$r'$}
\put(365,158){$r'$}
\end{picture}
\caption{On the left, a transversal picture where the central dot is~$k$, the stable manifolds are in blue, the unstable ones in red, the empty dots correspond to the homoclinic W-S connection~$q$, the pentagons to the orbit~$q'$, the triangles to the orbit~$p'$ and the 4-pointed stars to the orbit~$r'$. On the right the pair of pants~$P$ transverse to the vector field~$R$ and intersecting the orbit~$k$ constructed by Fried: it is the union of a parallelogram in a local transverse disc whose vertices lie on~$p', r', q'$ and $r'$ respectively, and the image under the flow of the two ``almost stable edges'' of this parallelogram. This topological surface is not strictly speaking transverse to the flow, but one can smooth it and make it transverse to~$R$ at the same time. 
}
\label{figure:FriedPants}
\end{figure}
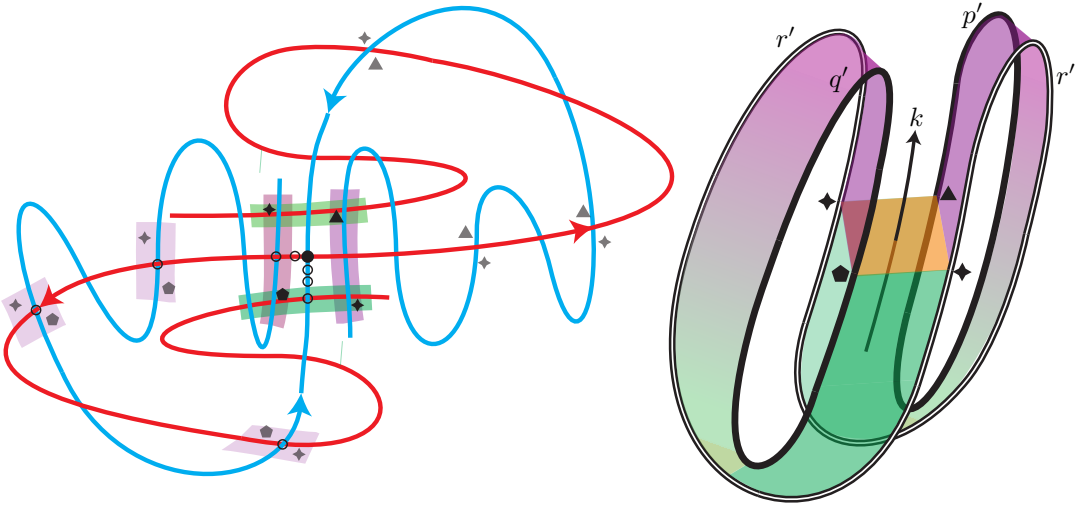

We add this section $P$ to our broken book decomposition, via the following process, sometimes called Fried sum: 
denoting by $S_{R_\F}$ the union of the rigid pages, we consider the union $P\cup S_{R_\F}$. 
It is an immersed surface with 1-dimensional singularities consisting of arcs ending on boundary orbits of $P$ or~$S_{R_\F}$, see Figure~\ref{figure:BindingSuppression} left.
We desingularize those arcs transversally to the vector field, as on the central picture. 
On a broken binding orbit transverse to~$P$, the result is depicted on Figure~\ref{figure:BindingSuppression} right: the addition of a meridian circle corresponding to the intersection with~$P$ to the boundary of~$S_{R_\F}$ gives a boundary for the new surface that intersects all orbits of the projectivized flow, including the stable and unstable direction of the broken binding orbit. 
We then obtain a new broken book decomposition, whose new binding is the initial one increased by the three new orbits $p', q', r'$ and whose broken binding is the previous with  $k$ removed, i.e. $k$ is no more a broken binding component anymore: it can either disappear from the binding or become a radial component, depending on the balance between the number of incoming and outgoing rigid pages. 
Since the orbits $p'$, $q'$ and $r'$ are intersecting the pages of $(K,\mathcal{F})$ transversally, they become radial components of the new binding.

\begin{figure}[ht]
\centering
\includegraphics[width=.4\textwidth]{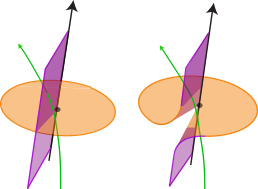}\qquad
\includegraphics[width=.25\textwidth]{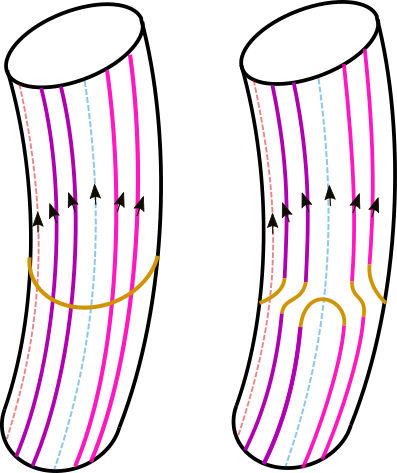}
\caption{
On the left is depicted the Fried sum of two transverse surfaces, which amounts to consider their union and desingularize it transversaly to the vector field. 
On the right what happens along a broken binding orbit~$k$ when adding the pair of pants~$P$ to the collection of rigid pages~$S_{R_\F}$ . 
The torus is the blow-up of~$k$. 
The stable and unstable manifolds are dotted, the rigid pages are in pink and purple depending on the sector, and~$P$ is the orange meridian circle. 
If the sum of the longitudinal coordinates of all boundary components of~$S_{R_\F}$ add up to~$0$, then the Fried union with $P$ also has zero longitudinal coordinate, hence the boundary of the new surface can be made meridional. 
After an isotopy, we obtain a surface with no boundary component along~$k$ and transverse to it. 
If this sum is nonzero, then the obtained surface still has~$k$ in its boundary, but $k$ is now in the radial part of the boundary ({\it i.e.} the flow winds with respect to the new surface).
}
\label{figure:BindingSuppression}
\end{figure}

By applying this process on all broken binding orbits successively, we reduce the broken part of the binding. 
After a finite number of steps, the broken book has only radial binding components, and all its pages are Birkhoff sections for the Reeb flow.

The case where $k$ is a negative hyperbolic periodic orbit is treated in the same manner. The difference is that now one needs to consider the second iterate of the return map to a local transversal to the periodic orbit in order to have the same picture.

The conclusion so far is that the proof of Theorem \ref{thm: section} boils down to proving that by a $C^1$-small perturbation supported away from the binding $K$, one can introduce transverse homoclinic connections at will. Note that if the perturbation is small enough and supported outside a fixed neighborhood of $K$, then the Reeb vector field remains carried by $(K,\mathcal{F})$ since the transversality condition is open away from the binding.

This is a direct application of the work of Bonatti-Crovisier \cite{BC} and Arnaud-Bonatti-Crovisier \cite{ABC} in the Reeb case.
The only thing to check is that Reeb flows satisfy the {\it Lift Axiom} of Pugh-Robinson~\cite{PR}. 
For its statement and proof below, we denote by $D_r(z) \subset \R^{2n}$ and $B_r(z) \subset \R^{2n}$ the closed and open euclidean balls of radius $r > 0$ and center $z \in \R^{2n}$, respectively.

\begin{lemma}[Lift Axiom]\label{lemma: lift Axiom}
Equip $\R^{2n} \times \R$ with coordinates $(z,t)$.
Consider a contact form $\lambda$ on $D_1(0) \times [0,1]$ with Reeb vector field $R = \partial_t$.
There exist $K>0$ and $0<\epsilon_*<1/2$ with the following property.
For every $0<\epsilon<\epsilon_*$ and~$z_0 \in B_\epsilon(0) \setminus \{0\}$ there exists a smooth contact form $\lambda'$ on $D_1(0) \times [0,1]$ with the same contact kernel than $\lambda$ such that $\|\lambda'-\lambda\|_{C^2} < K \epsilon$, $\lambda'-\lambda$ is compactly supported in the open box $B_{\epsilon^{-1}|z_0|}(0) \times (0,1)$, and the Reeb flow of $\lambda'$ takes $(0,0)$ to $(z_0,1)$.
\end{lemma}

\begin{proof}
Fix $\beta:[0,1]\to[0,1]$ smooth such that $\beta(t) = 0$ for $t$ near $0$, $\beta(t) = 1$ for $t$ near $1$, $\beta'\geq0$.
For every $z_0 \in B_1(0)$ consider the curve $\gamma(t;z_0) = (\beta(t)z_0,t)$.
We look for a family of contact Hamiltonians $h_{z_0}(z,t)$ on $D_1(0) \times [0,1]$, depending on a parameter $z_0 \in \R^{2n}$. 
The associated contact Hamiltonian vector field is $X = h_{z_0}R+Y$, where $Y$ is determined by $i_{Y}\lambda=0$, $dh_{z_0}-(i_Rdh_{z_0})\lambda = i_{Y}d\lambda$. 
We would like to achieve that $X$ is positively tangent to $\gamma$, and that $h_{z_0}=1$ along~$\gamma$.
A simple calculation shows that these constraints can be met if $|z_0|$ is small enough.
Moreover, if we represent $dh_{z_0}(\gamma(t)) = \left< V(t;z_0),\cdot \right>$ by a vector field then $|V|+|\partial_tV|+|\partial_t^2V|\leq c|z_0|$ for some $c>0$ independent of $z_0$.
Note that $\left<V(t;z_0),\partial_t\gamma(t;z_0)\right> = 0$ for all $t$.
Perhaps after making $c$ larger, but still independent of $z_0$, the function $\hat h_{z_0}(z,t) = 1 + \left< V(t;z_0),(z-\beta(t)z_0,0) \right>$ satisfies 
$$ 
\|\hat h_{z_0}-1\|_{C^2(D_1(0)\times[0,1])} \leq c|z_0|\, , \qquad |\hat h_{z_0}(z,t)-1| \leq c|z_0| (|z|+|z_0|)
$$
and meets the requirements along $\gamma$.
The discussion so far is independent of $\epsilon$.
Consider smooth bump functions $\phi_{z_0}^\epsilon(z)$ with values in $[0,1]$, $\phi_{z_0}^\epsilon = 1$ on $D_{2|z_0|}(0)$, $\supp(\phi_{z_0}^\epsilon) \subset B_{\epsilon^{-1}|z_0|}(0)$, $\|\nabla\phi_{z_0}^\epsilon\|_{\infty} = O(\epsilon|z_0|^{-1})$, $\|\nabla^2\phi_{z_0}^\epsilon\|_\infty = O(\epsilon^2|z_0|^{-2})$.
To conclude, note that there exists $\epsilon_*>0$ small enough such that if $0<\epsilon<\epsilon_*$ then $h_{z_0} = 1+\phi_{z_0}^\epsilon(z)(\hat h_{z_0}(z,t)-1)$ is the contact Hamiltonian we were looking for, and $\lambda' = h_{z_0}\lambda$.
\end{proof}

For the rest of the argument, recall from \cite{ABC,BC} that it uses the fact that in the conservative setting there are no wandering points and that there exists $N >0$ depending only on the $C^1$-size of the perturbation $\epsilon$, such that  given any fixed closed neighborhood $U$ of  finitely many periodic orbits one can join any two points in the same 
 connected component of $M\setminus (\cup_{t\in [0,N]} \phi_t(U))$ by a $\delta$-pseudo-orbit, where $\delta \ll \epsilon$ can be chosen small at will independently of $N$ and $\epsilon$.
The jumps of the pseudo-orbit are contained in a finite number of disjoint flow-boxes in $M\setminus U$ whose radii are of order $\delta$ and length $N$.
To construct the flow-boxes, $\delta$ has to be taken very small once $\epsilon$ and $N$ have been fixed.
  Here $\phi_t$ denotes the flow of $R$ at time $t$.  
 The $\delta$-pseudo-orbits can then be turned into genuine ones by $\epsilon$-$C^1$-small and $\delta$-$C^0$-small perturbations of the flow, supported in the flow-boxes (and thus in the complement of  $U$) and given by the Lift Axiom, see \cite{ABC,BC}. 
  
 Precisely here, to make sure that we create a homoclinic connection, we take a point $a$ on a unstable manifold of  a broken component of the binding $k$ and a point $b$ on one of its stable manifolds. The set $U$ is a small closed neighborhood of the binding of the supporting broken book so that $a$ and $b$ are in the boundary of $U$. We can moreover get that
 \begin{itemize}
 \item the orbit of $a$ for negative time stays in $U$ as well as the orbit of $b$ for positive times.
 \item the orbit of $a$ for time $[0,N]$ stays outside of $U$ as well as the orbit of $b$ for time $[-N,0]$.
 \end{itemize}

We can then find, following \cite{ABC,BC}, two first flow-boxes around the portions of orbits $\phi_{[0,N]}(a)$ and $\phi_{[-N,0]}(b)$ and the remaining ones away from the two first and $U$, provided the endpoints $\phi_N(a)$ and $\phi_{-N}(b)$ belong to the same connected component of $M\setminus (\cup_{t\in [0,N]} \phi_t(U))$. This last condition is always achievable by taking $U$ small enough.
 
Then the Lift Axiom gives a perturbed Reeb vector field, with a deformation supported in our collection of flow-boxes (and thus in $M\setminus U$),  whose flow connects~$a$ and $b$. 
This in turn yields a homoclinic orbit for $k$ passing through $a$ and $b$, since we did not perturb the portions of orbits in $U$ through $a$ and $b$ that were respectively negatively and positively asymptotic to $k$.
An extra $C^\infty$-small perturbation is enough to make the homoclinic connection transverse.
If $\delta$ is taken small enough, the transversality of the perturbed Reeb vector field with the pages of the broken book away from the binding is preserved.

In what follows $M$ is a closed contact manifold of dimension $2n+1$ for $n\geq 1$. To go slightly further, let us summarize the connecting lemma that is in particular obtained in the previous discussion:

\begin{theorem}[Arnaud, Bonatti and Crovisier \cite{ABC} $+\epsilon$]\label{thm: connecting}
Let $(M,\xi=\ker \lambda)$ be a closed contact manifold with $\lambda$ nondegenerate, then given any two points $x,y\in M$, there exists a $C^2$-small perturbation $\lambda'$ of $\lambda$ 
such that $y$ is in the positive orbit of $x$ under the flow of $R_{\lambda'}$.
\end{theorem}

From this result, following Arnaud-Bonatti-Crovisier \cite[Section 5.1]{ABC}, one obtains:

\begin{theorem}\label{thm: transitive} On a closed contact manifold, the set of transitive Reeb vector fields is a $G_\delta$-dense in the set of Reeb vector fields for the $C^1$-topology.
\end{theorem}
\begin{proof} The arguments of \cite[Section 5.1]{ABC} apply \textit{verbatim}. We reproduce them here for the reader's convenience.
First the set $\mathcal{G}_0$ of nondegenerate Reeb vector fields is a $G_\delta$.
We then consider a countable basis $(U_n)_{n\in \N}$ of neighborhoods of $M$ and for~$m,n\in \N$ we let $\mathcal{U}_{m,n}$ be the set of Reeb vector fields $R$ for which the flow $(\phi_t)_{t\in \R}$ satisfies: there exists $t>0$ with $\phi_t (U_n)\cap U_m \neq \emptyset$.
The set $\mathcal{U}_{m,n}$ is open in the $C^1$-topology. Moreover $\mathcal{G}_0 \cap \mathcal{U}_{m,n}$ is $C^1$-dense by the connecting lemma (Theorem~\ref{thm: connecting}) and so is $\mathcal{U}_{m,n}$.
Thus
$$\mathcal{G}=\bigcap_{m,n} \mathcal{U}_{m,n}$$ is a $G_\delta$.

Every $R$ in $\mathcal{G}$ is by definition topologically transitive and thus transitive: it has a dense orbit.
Indeed, if $R\in \mathcal{G}$, for every $m\in \N$ the set $V_m$ of points of $M$ whose positive orbit meets $U_m$ is open and dense.
Thus $\bigcap_{m} V_m$ is a $G_\delta$. 
Every point $q \in \bigcap_{m} V_m$ is on a dense orbit.
\end{proof}

\end{document}